\documentclass[11pt, reqno]{article}
\providecommand{\keywords}[1]{\textbf{\textit{Key words and phrases---}} #1}

\usepackage{amsfonts}
\usepackage{amscd}
\usepackage{amsmath,amssymb}
\usepackage{enumerate}
\usepackage{graphicx}
\usepackage{accents}
\usepackage{comment}
\newlength{\dhatheight}

\usepackage{bbm}
\usepackage{enumerate}
\usepackage{amstext}
\usepackage{mathrsfs}
\usepackage{xspace}
\usepackage{amstext}
\usepackage{amsthm}       
\usepackage{mathrsfs}
\usepackage{xspace}
\usepackage[active]{srcltx}
\usepackage{cite}
\usepackage{enumitem}

\usepackage[colorlinks,
linkcolor=red,
anchorcolor=green,
citecolor=blue
]{hyperref}

\newcommand{\CC}{\D{C}}

\newcommand{\R}{\mathbb R}

\newcommand{\Z}{\mathbb Z}

\newcommand{\supp}{\operatorname{supp}}

\newcommand{\Mis}{\mathfrak{M}}

\newcommand{\dist}{\mathrm{dist}}

\newcommand{\tagliato}{$\kern-5 mm -$}
\newcommand{\tagliat}{$\kern-4 mm -$}

\newcommand{\D}[1]{\mbox{\rm #1}}
\newcommand{\dd}{{\rm d}}

\newtheorem{thm}{Theorem}[section]
\newtheorem{prop}[thm]{Proposition}
\newtheorem{lemme}[thm]{Lemma}
\newtheorem{definition}[thm]{Definition}
\newtheorem{cor}[thm]{Corollary}
\theoremstyle{definition}

\newtheorem{rem}[thm]{Remark}
\newtheorem{rems}[thm]{Remarks}

\numberwithin{equation}{section}

\newcommand{\galax}{\gamma^x_\lambda}
\newcommand{\dgalax}{\dot{\gamma}^x_\lambda}

\def\red#1{\textcolor{red}{#1}}

\title{Convergence of the solutions\\ 
of the \red{nonlinear} discounted Hamilton--Jacobi equation:
The central role of Mather measures.}

\author{Qinbo Chen\footnote{\scriptsize Department of Mathematics, Nanjing University, Nanjing 210093, China.
{\it E-mail address:} {\tt qinbochen1990@gmail.com}}
\and 
Albert Fathi\footnote{\scriptsize Georgia Institute of Technology, 
School of Mathematics, Atlanta, GA 30332, USA. 
{\it E-mail address:} {\tt albert.fathi@math.gatech.edu}
Work finalized in December 2022, while at Equipe d'Analyse Alg\'ebrique  IMJ Paris
}
\and 
Maxime Zavidovique\footnote{\scriptsize Sorbonne Universit\' e, Universit\'e de Paris Cit\' e, CNRS, Institut de Math\' ematiques de Jussieu- Paris Rive Gauche, F-75005 Paris, France.
{\it E-mail address:} {\tt maxime.zavidovique@imj-prg.fr}
Supported by ANR CoSyDy   (ANR-CE40-0014)}
\and 
Jianlu Zhang\footnote{\scriptsize Hua Loo-Keng Key Laboratory of Mathematics and Mathematics Institute, Academy of Mathematics and Systems Science, Chinese Academy of Sciences, Beijing, 100190, China.  
{\it E-mail address:} {\tt jellychung1987@gmail.com}}
}

\date{}

\begin{document}

\maketitle

\vskip -7cm
\centerline{\hfill Dedicated to the memory of John Mather (1942--2017)}

\vskip 8cm


\begin{abstract}
Given a continuous Hamiltonian $H : (x,p,u) \mapsto H(x,p,u)$ defined on 
$ T^*M \times \mathbb R $, where $M$ is a closed connected manifold, we study viscosity solutions, $u_\lambda : M\to \mathbb R$, of  discounted equations:
$$ H(x, d_x u_\lambda, \lambda u_\lambda(x))=c\ \ \textup{in $M$}$$
where $\lambda >0$ is called a discount factor and $c$ is the critical value of $H(\cdot, \cdot , 0)$.

When $H$ is convex and superlinear in $p$ and non--decreasing in $u$, under an additional non--degeneracy condition, we obtain existence and uniqueness (with comparison principles) results of solutions and we prove that the family of solutions $(u_\lambda)_{\lambda >0}$ converges to a specific solution $u_0$ of  
$$ H(x, d_x u_0, 0)=c\ \ \textup{in $M$}.$$
Our degeneracy condition requires $H$ to be increasing (in $u$) on localized regions linked to the support of Mather measures, whereas usual similar results are obtained for Hamiltonians that are everywhere increasing in $u$. 
\end{abstract}

\keywords{Discounted  Hamilton--Jacobi equations, viscosity solutions,  comparison principle, weak KAM Theory, Mather measures}

 2010 \emph{ Mathematics Subject Classification.}  \ 35B40, 35F21, 37J50, 49L25.

\section{Introduction}\label{section_Intr}
Let $M$ be a closed connected smooth manifold. In the sequel, $H: T^*M\times\R\to \R$ will be a given continuous function, called the {\em Hamiltonian}, where $T^*M$ is the cotangent bundle of $M$. We denote by $(x,p,u)$ points of $T^*M \times \mathbb R$. We will assume that $H$ is convex and superlinear in the variable $p$.

We are interested in the following approximation problem for Hamilton--Jacobi equations of the form
\begin{equation}\label{intro_eq}
	H(x, d_x u_\lambda, \lambda u_\lambda(x))=c\ \ \textup{in $M$}
\end{equation}
where $\lambda>0$ is called a discount factor in view of stochastic optimal control and differential games. Note that in the limit case $\lambda=0$, there is a unique value $c\in \R$, the so-called {\em critical value}, such that
\begin{equation}\label{intr_criteq}
	H(x, d_xu, 0)=c \ \ \textup{in $M$}
\end{equation} 
admits solutions, so we will use this value $c$ in \eqref{intro_eq} as well.  The notion of {\em solution, subsolution and supersolution}  adopted in this paper is in the viscosity sense \cite{Crandall_Lions1983}, see also the books \cite{Barles_book, Bardi_Capuzzo-Dolcetta1997}.

Our goal in this paper is to explore the asymptotic behavior, as $\lambda\to 0^+$, of the viscosity solutions of \eqref{intro_eq} in a setting where $H(x,p,u)$ can be nonlinear in $u$, and not necessarily strictly increasing in $u$ everywhere. Such equations will be called {\em nonlinear degenerate discounted H-J equations}. Note that this asymptotic problem is not easy because the critical equation \eqref{intr_criteq} has infinitely many solutions, even up to additive constants in general. As will see later, Mather (minimizing) measures play a central role in the convergence of $u_\lambda$ (as $\lambda\rightarrow 0^+$).

Our approximation problem originates from the seminal paper by Lions, Papanicolaou and Varadhan \cite{Lions_Papanicolaou_Varadhan1987} in 1987  on homogenization of Hamilton--Jacobi equations. In this work, they considered the {\em vanishing discount problem} (also called {\em ergodic approximation})
\begin{equation}\label{Discounted_equation}
	\lambda u_\lambda(x)+G(x, d_x u_\lambda)=0 \ \ \textup{in $M$}
\end{equation} 
and discussed the asymptotic behavior of the solutions as $\lambda\to 0^+$. This is of course a particular case of our setting by choosing $H(x,p,u)=u+G(x,p)$. Equation \eqref{Discounted_equation} obeys the comparison principle and admits a viscosity solution denoted by $u_\lambda$. They pointed out $\lim_{\lambda\rightarrow 0_+}\lambda u_\lambda=-c(G)$ with $c(G)$ being the critical value of the system $G$. Therefore, the next step is to investigate the convergence of $u_\lambda+c(G)/\lambda$ as $\lambda\rightarrow 0^+$ as well. It turns out that $\{u_\lambda+c(G)/\lambda\}_{\lambda>0}$ are equibounded and equi-Lipschitz, thus, by the Ascoli-Arzel\`a theorem and the stability property of viscosity solutions \cite{Crandall_Evans_Lions1984}, they converge, \emph{along subsequences} as $\lambda$ goes to $0$, to viscosity solutions of the critical equation 
\begin{equation}\label{criteq_G}
	G(x,d_x u)=c(G) \ \ \textup{in $M$.}
\end{equation}
At that time, it was not clear if different subsequences yield the same limit solution of \eqref{criteq_G}. This problem remained unsolved for quite some time.  Some constraints on the possible limit solutions were found in \cite{Gomes2008, Iturriaga_Sanchez-Morgado2011}. The essential breakthrough was made by \cite{DFIZ2016}, where they established the convergence to a unique limit solution as well as the characterization of the limit, by using tools from weak KAM theory. We also refer to \cite{Maro_Sorrentino2017, Arnaud_Su} for geometric interpretation of the convergence.

The result of asymptotic convergence has been subsequently generalized to the second order HJ setting \cite{Mitake_Tran2017,Ishii_Mitake_Tran2017_1, Ishii_Mitake_Tran2017_2}. Nowadays, this topic is still very active and progress has been made in many different directions. We mention some here: \cite{DFIZ_Math_Z_2016} for a discrete time version, \cite{Cardaliaguet_Porretta_MFG19} for mean field games,
\cite{CCIZ2019, Gomes_Mitake_Tran2018, WYZ_2021, QC} for contact Hamilton--Jacobi versions, \cite{Davini_Siconolfi_Zavidovique_18, Davini_Zavidovique2017, Ishii2019vanishing, Ishii_Jin2020} for weakly coupled systems, \cite{Ishii_Siconolfi2020, Davini2022} for a noncompact setting, \cite{Davini_Wang2020,WYZ_negative} for HJ equations with negative discount factors. 

We point out that all papers mentioned above required a strict monotonicity hypothesis. A recent work \cite{MZ} first studied some degenerate HJ equations of the form
\begin{equation}\label{intro_dg}
	\lambda \alpha(x) u_\lambda(x)+G(x, d_xu_\lambda)=c(G) \ \ \textup{in $M$,}
\end{equation}
where $\alpha: M\to \R$ is non-negative and may vanish at some places. The author proved that the convergence still holds if $\alpha>0$ on the projected Aubry set. This sheds some light on dealing with the convergence of solutions for systems with degeneracy.  

In the present paper, we will establish an asymptotic convergence result for nonlinear discounted HJ equations \eqref{intro_eq} with degenerate discounting. We will show the importance of Mather measures in controlling the convergence of solutions. As merely a subset of the Aubry set, the Mather set allows a better understanding of systems, see our condition {\bf(L4)}. As far as we know such a condition {\bf(L4)} is new in the literature. This condition may be optimal for our convergence problem (see Section \ref{section_counterexample} for a counterexample showing some optimality), and in addition, we also discuss how to use {\bf(L4)} to obtain a generalized comparison principle. As the Hamiltonians $H(x,p,u)$ involved in our equation \eqref{intro_eq} depends implicitly or nonlinearly on the unknown function $u$, there is in general no explicit representation formula for the value function $u_\lambda$ to equation \eqref{intro_eq}. To overcome this difficulty, 
we borrow tools from weak KAM theory and use arguments that do not rely on explicit representation formulas of the value functions. That is another novelty comparing to previous works such as \cite{DFIZ2016, MZ} and provides a chance to further apply our results.

\subsection{Setting}
Let $M$ be a closed connected smooth manifold endowed with a Riemannian metric. As $M$ is compact, all such Riemannian metrics are equivalent and all our results are independent of this choice. Let $T^*M$ and $TM$ be the cotangent and tangent bundles of $M$, respectively. For $x\in M$, we will use  $\lVert\cdot\rVert_x$  to denote the norm associated to $g$, either on $T^*M$ or on $TM$.

On this compact manifold $M$, we consider a continuous Lagrangian $L(x,v,u): TM\times \R\to \R$ which satisfies the following conditions:
\begin{itemize}[itemindent=1em]
\item[\bf(L0)] (Monotonicity) For every given $(x,v)\in TM$, $L(x,v,u)$ is non-increasing in 
$u\in\R$. 
\item[\bf(L1)] (Convexity) For every given $x\in M$ and $u\in \R$,  $L(x,v,u)$ is convex in $v\in T_xM$.
\item[\bf(L2)] (Superlinearity) For every $K,u_0\in [0,+\infty[ $, we have
$$\sup\{ K \lVert v\rVert_x-L(x,v,u)\mid (x,v,u)\in TM\times [-u_0,u_0]\}<+\infty.$$
\end{itemize}
Note that condition {\bf(L0)} implies $L(x, v, u) \geq L(x, v, u_0)$ for all $u\leq u_0$, so it suffices to assume the superlinearity for $L(x, v, u_0)$ in {\bf(L2)}.

By Fenchel's formula, the Lagrangian $L$ has a conjugated Hamiltonian $H: T^*M\times \R\to \R$, i.e.,
\begin{equation}\label{def H}
H(x,p,u):=\sup_{v\in T_xM} p(v) -
L(x,v,u),\text{ for every $(x,p,u)\in T^*M\times \R$}.
\end{equation}
The Hamiltonian $H$ is continuous, and conditions {\bf(L0)}--{\bf(L2)} on $L$ directly lead  to the following conditions on $H$:
\begin{itemize}[itemindent=1em]
\item[\bf (H0)] (Monotonicity) For every given $(x,p)\in T^*M$, $H(x,p,u)$ is non--decreasing in 
$u\in\R$. 
\item[\bf(H1)] (Convexity) For every given $x\in M$ and $u\in\R$, $H(x,p,u)$ is convex in $p\in T^*_xM$.
\item[\bf(H2)]  (Superlinearity) For every $K,u_0\in [0,+\infty[ $, we have
$$\sup\{ K \lVert p\rVert_x-H(x,p,u)\mid (x,p,u)\in T^*M\times [-u_0,u_0]\}<+\infty.$$
\end{itemize}

For brevity, we can introduce the Lagrangian $L^r:TM\to \R, (x,v)\mapsto  L(x,v,r)$ for each fixed $r\in \R$, and its associated Hamiltonian $H^r:T^*M\to\R$, $(x,p)\mapsto H(x,p,r)$. Accordingly, we denote by $c(H^r)$ the critical value of $H^r$.


We aim to consider Hamilton--Jacobi equations of the form
\begin{equation}\label{HJlambda}
H(x,d_xu_\lambda,\lambda u_\lambda(x))=c(H^0) \quad \textup{in $M$}\tag{HJ$_\lambda$}
\end{equation}
where $\lambda>0$. We are interested in understanding the asymptotic behavior of the global continuous viscosity solutions $u_\lambda : M\to \R$ as $\lambda$ goes to zero.

To this aim, we now introduce another condition {\bf(L3)} on the Lagrangian $L$.
\begin{itemize}[itemindent=1em]
 \item[\bf(L3)] ($u$-derivative at $0$) 
 For every $(x,v)\in TM$, the partial derivative $\partial L/\partial u(x,v,0)$ exists. Moreover, for every compact subset $S\subset TM$, we can find  a modulus of continuity $\eta_S$ such that
 $$\left\lvert L(x,v,u)-L(x,v,0)-u\frac{\partial L}{\partial u}(x,v,0)\right\rvert\leq \lvert u\rvert \eta_S(\lvert u\rvert),
 \text{ for all $(x,v,u)\in S\times \R.$}$$
\end{itemize}
Recall that  a modulus of continuity is a non--decreasing function $\eta_S:[0,+\infty[\to[0,+\infty[$ with $\lim_{u\to 0}\eta_S(u)=0$. 

Actually, condition {\bf(L3)} is quite general, in view of the fact that it holds whenever $L$ is $C^1$. 

In order to formulate the last condition, we will need to use the notion of Mather measure (see Section \ref{sectionNon-u}) for the critical Lagrangian $L^0$. We denote by $\Mis(L^0)$ the set of all Mather measures of $L^0$.
\begin{itemize}[itemindent=1em] 
 \item[\bf(L4)] (Non--degeneracy on Mather set) 
 $$\int_{TM}\frac{\partial L}{\partial u}(x,v,0)\,\dd\tilde\mu(x,v)<0,\text{ for all Mather measures 
 $\tilde\mu\in\Mis(L^0)$}.
 $$
\end{itemize}

Such an integral form of monotonicity condition w.r.t. $u$ concerning Mather measures is new, as far as we know. It is crucial in the study of the asymptotic behavior of solutions, and is also the key to obtaining a comparison principle for equation \eqref{HJlambda} (see  Section \ref{section_uniqueness} for more discussion about {\bf(L4)}).

\subsection{Main results}

\subsubsection{Convergence of the solutions of  \texorpdfstring{\eqref{HJlambda}}{}}
Let us first present the following result:

\begin{thm}\label{mythm0}
	If the Lagrangian $L:TM\times \R\to\R$ satisfies conditions {\bf (L0)}, {\bf(L1)}, {\bf(L2)}, {\bf(L3)} and {\bf(L4)}, then for every $\lambda>0$  equation \eqref{HJlambda} admits global continuous viscosity solutions. 
	
	Moreover, the viscosity solutions of \eqref{HJlambda}, with $\lambda>0$, are equibounded. Namely, we can find a constant $\mathfrak{C}>0$ such that for any global continuous viscosity solution $u_\lambda$ of \eqref{HJlambda} and any $\lambda>0$,
	\begin{equation*}
		\lVert u_\lambda\rVert_\infty\leq \mathfrak {C}.
	\end{equation*}
	  In addition, for every given $\lambda_0>0$, the family of viscosity solutions of \eqref{HJlambda} for all $\lambda\in ]0, \lambda_0]$, is equi-Lipschitz. 
	\end{thm}

Since the Hamiltonian $H(x,p,u)$ is not necessarily strictly increasing in $u$,  it is not clear whether equation \eqref{HJlambda} with $\lambda >0$ obeys a comparison principle. Nonetheless, thanks to condition {\bf(L4)}, the solutions $u_\lambda$'s of \eqref{HJlambda} are equibounded for all $\lambda>0$. Based on Theorem \ref{mythm0}, we give the main result on the asymptotic behavior, as $\lambda\to 0^+$, of the viscosity solutions of \eqref{HJlambda} as the following:

\begin{thm}\label{main} If the Lagrangian $L:TM\times \R\to\R$ satisfies conditions {\bf (L0)}, {\bf(L1)}, {\bf(L2)}, {\bf(L3)} and {\bf(L4)}, then there exists a continuous function $u_0:M\to \R$ such that:
\begin{enumerate}
\item The family of all  possible viscosity solutions $(u_\lambda)_{\lambda>0}$ of \eqref{HJlambda} converges uniformly to $u_0$ as $\lambda\to 0$.
\item The particular function $u_0$ is  a viscosity solution of the critical Hamilton--Jacobi equation
\begin{equation}\label{HJ000} H^0(x,d_x u)=c(H^0) \quad \textup{in $M$}.
\end{equation}
The function $u_0:M \to \R$ is also the largest viscosity subsolution $w$ of \eqref{HJ000} satisfying 
\[\int_{TM}w(x)\frac{\partial L}{\partial u}(x,v,0)\,\dd\tilde \mu(x,v)\geq0,\ \  \text{for all Mather measures $\tilde\mu\in\Mis(L^0)$.}\]	
\end{enumerate}
\end{thm}

\begin{rem}
For each $\lambda>0$, let us denote by $\mathcal{F}_\lambda$ the set of all continuous viscosity solutions of equation \eqref{HJlambda}. Then Theorem \ref{main} implies that 
\begin{equation*}
	\dist_H(\mathcal{F}_\lambda, \{u_0\} )\to 0, \quad \textup{as $\lambda\to 0^+$}
\end{equation*}
where $\dist_H(\cdot,\cdot)$ denotes the Hausdorff distance between two sets in $C^0(M,\R)$, or equivalently 
$$\sup_{w_\lambda\in \mathcal{F}_\lambda}\|w_\lambda-u_0\|_\infty  \to 0, \quad \textup{as $\lambda\to 0^+$}.$$
\end{rem}

We also provide another formula for the limit $u_0$. This formula involves the Peierls barrier (see Section \ref{sectionNon-u}) and Mather measures.

\begin{thm}\label{mythm_anotherformula}
The limit solution $u_0:M\to \R$ obtained in Theorem \ref{main} satisfies 
\begin{equation*}\label{In_Math_meas}
		u_0(x)=\inf_{\tilde{\mu}\in \Mis(L^0) }\frac{\int_{TM}h(y,x)\frac{\partial L}{\partial u}(y,v,0)\,\dd\tilde{\mu}(y,v)}{\int_{TM}\frac{\partial L}{\partial u}(y,v,0)\,\dd\tilde{\mu}(y,v)}.
\end{equation*}	
where $h:M\times M\to \R$ is the Peierls barrier of the Lagrangian $L^0$.
\end{thm}

\begin{rems} We now give some explanations about our hypotheses. 
\begin{enumerate}
	\item Our analysis strongly relies on the convexity condition {\bf(L1)} or equivalently {\bf(H1)}, because we will use a dynamical systems approach in light of weak KAM theory for convex superlinear systems. We also refer to \cite{Ziliotto_counterexample_2019} for a remarkable counterexample showing that without the convexity the family of solutions $u_\lambda$ do not converge. 
	
	\item By the superlinearity condition {\bf(L2)}, if we set
$-C(K,u_0)=\sup\{ K \lVert v\rVert_x-L(x,v,u)\mid (x,v,u)\in TM\times [-u_0,u_0]\},$
then $C(K,u_0)$ is finite and
\begin{equation}\label{SuperLinL}
 L(x,v,u)\geq K \lVert v\rVert_x+C(K,u_0),\text{ for all $ (x,v,u)\in TM\times [-u_0,u_0]$.}
\end{equation}
As above, by condition {\bf(H2)} if we set
$-C^*(K,u_0)=\sup\{ K \lVert p\rVert_x-H(x,p,u)\mid (x,p,u)\in T^*M\times [-u_0,u_0]\},$
we have $C^*(K,u_0)$ is finite and
$$H(x,p,u)\geq K \lVert p\rVert_x+C^*(K,u_0),\text{ for all $ (x,p,u)\in T^*M\times [-u_0,u_0]$.}$$
By standard arguments as in \cite{DFIZ2016, CCIZ2019}, hypothesis {\bf(H2)} can also be weakened to a coercivity assumption: 
\begin{itemize}[itemindent=1em]
 \item [\bf(H2$'$)] For every given $r\in \R $,  $H(x,p,r)\to+\infty$ as $\|p\|_x\to \infty$.
\end{itemize}
Note that due to the compactness of $M$ and the convexity of $H$ in $p$, it is equivalent to require that coercivity holds pointwise for every $x\in M$ or uniformly on $M$.

   \item Condition {\bf(L3)} always holds if $L\in C^1$. Also, it is not difficult to show that {\bf(L3)} is equivalent to the following one:
   \begin{itemize}[itemindent=1em]
 \item For every $(x,v)\in TM$, the partial derivative $\partial L/\partial u(x,v,0)$ exists. Moreover, for every compact subset $S\subset TM$, we have
 $$\sup_{(x,v)\in S} \left\lvert\frac { L(x,v,u)-L(x,v,0)}u-\frac{\partial L}{\partial u}(x,v,0)\right\rvert\to 0, 
 \text{ as $u\to 0$.}$$
 \end{itemize} 
 
    \item Condition {\bf(L4)} plays essential r\^ole in our proof.  A counterexample will be given in Section \ref{section_counterexample} to show its necessity.

\end{enumerate}
\end{rems}

\subsubsection{Local and Global uniqueness of solutions to equation  \texorpdfstring{\eqref{HJlambda}}{}}

One may ask if there is a comparison principle for equation \eqref{HJlambda} with $\lambda>0$. We first give a uniqueness result for \eqref{HJlambda} with $\lambda$ being small, by slightly strengthening condition {\bf(L3)} to the following one:
\begin{itemize}[itemindent=1em] 
 \item[\bf(L3$'$)]  For every $(x,v,u)\in TM\times \R$, the partial derivative $\partial L/\partial u(x,v,u)$ exists
 and is continuous on $TM\times \R$\footnote{As will be clear later, this condition can be weakened to only hold on a neighborhood of $TM\times \{0\}$.}.
 \end{itemize}

\begin{thm}[Local uniqueness]\label{mythm_partial} Suppose that the Lagrangian  $L:TM\times \R\to\R$ satisfies conditions {\bf (L0)}, {\bf(L1)}, {\bf(L2)}, {\bf(L3$'$)} and {\bf(L4)}. Then there exists $\lambda_1>0$ such that the Hamilton--Jacobi equation \eqref{HJlambda} has a unique viscosity solution for $0<\lambda\leq \lambda_1$. 
\end{thm}

To address the uniqueness structure for all $\lambda>0$, we can introduce the following condition:
 \begin{itemize}[itemindent=1em]
 \item[\bf(L5)] (Concavity in $u$) For every given $(x,v)\in TM$, the function $ u \mapsto L(x,v,u)$  is concave\footnote{This condition is equivalent to: for every given $(x,p)\in T^*M$, the function $ u \mapsto H(x,p,u)$  is convex.}.
  \end{itemize}
  
  \begin{thm}[Global uniqueness] \label{mythm_global}Suppose that the Lagrangian  $L:TM\times \R\to\R$ satisfies conditions {\bf (L0)}, {\bf(L1)}, {\bf(L2)}, {\bf(L3)}, {\bf (L4)} and {\bf(L5)}. Then for every $\lambda>0$  the Hamilton--Jacobi equation \eqref{HJlambda} has a unique viscosity solution. 
\end{thm}

Finally, we conclude with the following corollary, as a generalization of the result in \cite{MZ}.
\begin{cor}
Let $\alpha:M\to \R$ be a non-negative continuous function, and $G(x,p): T^*M\to \R$ be a continuous Hamiltonian that is convex and superlinear in $p$. If 
\begin{equation*}
	\int_{M} \alpha(x)\, \dd\mu>0,\ \ \textup{for every projected Mather measure $\mu$ of $L_G$,}
\end{equation*}
where $L_G$ is the Lagrangian associated to $G$. Then, the following equation 
\begin{equation*}
	\lambda \alpha(x) u(x) + G(x, d_x u)=c(G) \quad \textup{in $M$} 
\end{equation*}
has a unique continuous viscosity solution $u_\lambda: M\to \R$ for $\lambda>0$. Moreover, The family $u_\lambda$ uniformly converges, as $\lambda\to 0$, to a single solution to equation $G(x, d_x u)=c(G)$.
\end{cor}

Here, the uniqueness part follows from Theorem \ref{mythm_global} and the convergence part follows from Theorem \ref{main}.

\subsection{Organization of the paper}

The proofs of our results mainly use a dynamical systems approach in light of weak KAM theory, combined with some ideas from \cite{MZ, WYZ_2021, QC}.

In Section \ref{sectionNon-u}, we recall some basic facts on viscosity solutions of the Hamilton--Jacobi equations, and some useful results and tools from weak KAM theory or Aubry-Mather theory. In Section \ref{section_bounded} we reveal properties of the viscosity solutions for specific equation \eqref{HJlambda}. In particular, we will use that any solution of \eqref{HJlambda} is in fact a solution of certain non-discounted equation, to which we can directly apply the results in Section \ref{sectionNon-u}. In Section \ref{section_a_lemma} we prove Lemma \ref{Gen-equibounded}, which plays an important role in obtaining lower and upper bounds for the family of solutions $\{u_\lambda\}_{\lambda>0}$, in the subsequent Section \ref{section_equibd}. With such a boundness estimate we then derive in Section \ref{section_equiLip} an equi-Lipschitz property for the family $\{u_\lambda\}$ as well as their calibrated curves. 

In Section \ref{section_mainproof} we deal with the crux of the proof for Theorem \ref{main}. In view of conditions {\bf(L3)} and {\bf (L4)} we will establish several technical lemmas about the $u_\lambda$-calibrated curves and their asymptotic properties. It allows us to introduce a family of Borel probability measures associated with the $u_\lambda$-calibrated curves, and in particular, we obtain a selection principle for Mather measures. Such a kind of asymptotic analysis plays a crucial role in proving the convergence.

Section \ref{section_uniqueness} is devoted to giving a comparison result for the HJ equation \eqref{HJlambda}.  Using Lemma \ref{Gen-equibounded} again and combining with condition {\bf (L4)}, we prove Theorem \ref{mythm_partial} and Theorem \ref{mythm_global}. In Section \ref{section_an_extension} we discuss the feasible extension of our convergence result. We also provide a counterexample in Section \ref{section_counterexample} to show some optimality of condition {\bf(L4)}.

\section{Preliminaries on \texorpdfstring{$u$}{}-independent Lagrangians and Hamiltonians}\label{sectionNon-u}
We will need results from weak KAM theory of convex autonomous Lagrangians and Hamiltonians independent of  $u\in\R$. These results   in the classical Aubry-Mather theory and weak KAM theory were originally established,
see \cite{Mather1991, Mather1993, Fathi_book}, for Tonelli Lagrangians or Hamiltonians,  which are by definition C$^2$. As is well known, they can be generalized  to continuous Lagrangians or Hamiltonians, due to the arguments in  \cite{Barles_book, Fathi_Siconolfi2004, Fathi_Siconolfi2005,  Fathi2012, DFIZ2016} and references therein.

Let us consider a continuous Lagrangian ${\mathfrak L}:TM\to\R$ which satisfies:
\begin{itemize}[itemindent=1em]
\item[(${\mathfrak L}1$)] (Convexity) For every given $x\in M$, the Lagrangian 
${\mathfrak L}(x,v)$ is convex in $v\in T_xM$.
 \item[(${\mathfrak L}2$)]  (Superlinearity) For every $K\in [0,+\infty[ $, we have
 $$\sup\{ K \lVert v\rVert_x-{\mathfrak L}(x,v)\mid (x,v)\in TM\}<+\infty.$$
\end{itemize}

The Hamiltonian ${\mathfrak H}:T^*M\to\R$ associated to the Lagrangian   ${\mathfrak L}:TM\to\R$ is 
$${\mathfrak H}(x,p):=\sup_{v\in T_xM} p(v) -
{\mathfrak L}(x,v),\text{ for every $(x,p)\in T^*M$}.
$$
This Hamiltonian ${\mathfrak H}$ is continuous and satisfies:
\begin{itemize}[itemindent=1em] 
\item[(${\mathfrak H}1$)] (Convexity) For every given $x\in M$, the Hamiltonian ${\mathfrak H}(x,p)$ 
is convex in $p\in T^*_xM$.
 \item[(${\mathfrak H}2$)]  (Superlinearity) For every $K\in [0,+\infty[ $, we have
 $$\sup\{ K \lVert p\rVert_x-{\mathfrak H}(x,p)\mid (x,p)\in T^*M\}<+\infty.$$
\end{itemize}

Observe that for the $u$-dependent Lagrangian  $L:TM\times \R\to\R$ satisfying conditions {\bf(L0)}--{\bf(L2)}, any Lagrangian 
$L^r:TM\to\R, \ (r\in\R)$, defined by $L^r(x,v)=L(x,v,r)$, satisfies (${\mathfrak L}1$) and (${\mathfrak L}2$). Moreover, the Hamiltonian associated to the Lagrangian $L^r$ is of course $H^r:T^*M\to\R, (x,p)\mapsto
H(x,p,r)$, and $H^r$ satisfies (${\mathfrak H}1$) and (${\mathfrak H}2$).

An important fact is that viscosity subsolutions are necessarily Lipschitz. See for instance \cite{Barles_book} or \cite[Section 5]{Fathi2012}.
More precisely, we have:

\begin{lemme}\label{Lip0} Suppose that  ${\mathfrak H}:T^*M\to\R$ is a continuous Hamiltonian satisfying condition (${\mathfrak H}2$).
Then for any given $c\in\R$, the set of viscosity subsolutions $u$ of
$${\mathfrak H}(x,d_xu)=c$$ 
is equi-Lipschitz with common Lipschitz constant 
$$K_c=\sup\{\lVert p\rVert_x\mid {\mathfrak H}(x,p)\leq c\}.$$
\end{lemme}

$K_c$ is of course finite by condition (${\mathfrak H}2$). The superlinearity condition (${\mathfrak H}2$) and the compactness of $M$ yield the following theorem (see for instance \cite{Lions_Papanicolaou_Varadhan1987, Fathi2012}). 
\begin{thm}\label{ExistenceViscHJSansu}  
Suppose that  ${\mathfrak H}:T^*M\to\R$ is a Hamiltonian satisfying condition (${\mathfrak H}2$). There exists a unique constant
$c=c({\mathfrak H})$, namely the critical value, such that the Hamilton--Jacobi equation 
\begin{equation}\label{HJSansu}
{\mathfrak H}(x,d_xu)=c
\end{equation} has a continuous global 
viscosity solution $u: M\to\R$. 

This constant $c({\mathfrak H})$ is the smallest constant $c\in \R$ such that \eqref{HJSansu} has a continuous 
viscosity subsolution, and also the largest  constant $c\in \R$ such that \eqref{HJSansu} has a continuous 
viscosity supersolution.
\end{thm}

We will also need the following results, see for instance \cite{Barles_book}.
\begin{prop}\label{subsol_properties}
	Suppose that ${\mathfrak H}:T^*M\to\R$ satisfies condition (${\mathfrak H}1$). Let $u\in C(M)$, then
	\begin{itemize}
		\item[\rm 1)] If $u$ is a pointwise supremum (resp. infimum) of a family of subsolutions (resp. supersolutions) of \eqref{HJSansu}, then $u$ is a subsolution (resp. supersolution) of \eqref{HJSansu}.
		\item[\rm 2)] If $u$ is a pointwise infimum of a family of equi-Lipschitz subsolutions of \eqref{HJSansu}, then $u$ is also a subsolution of \eqref{HJSansu}.
		\item[\rm 3)] If $u$ is a convex combination of a family of equi-Lipschitz subsolutions of \eqref{HJSansu}, then it is also a subsolution of \eqref{HJSansu}.
	\end{itemize}
\end{prop}

There is a definition of the viscosity solution using only the Lagrangian $\mathfrak L$ as is stated in the following theorem. We refer to \cite{Fathi_Siconolfi2005, Davini_Zavidovique2013,DZ15} for its proof.
\begin{thm}\label{calibrationThm} Suppose that  ${\mathfrak H}:T^*M\to\R$ is a Hamiltonian satisfying conditions 
(${\mathfrak H}1$) and (${\mathfrak H}2$), whose associated Lagrangian is $\mathfrak L:TM\to\R$. 
\begin{itemize}
\item[\rm 1)] A function $u\in C(M)$ is a viscosity subsolution of
${\mathfrak H}(x,d_xu)=c({\mathfrak H})$ if and only if it satisfies: for every absolutely continuous curve $\gamma:[a,b]\to M$, we have
\begin{equation}\label{domination}
u(\gamma(b))-u(\gamma(a))\leq 
\int_a^b\mathfrak L(\gamma(s),\dot\gamma(s))+c(\mathfrak H)\,\dd s.
\end{equation}
\item[\rm 2)]  A function $u\in C(M)$ is a viscosity solution of ${\mathfrak H}(x,d_xu)=c({\mathfrak H})$ if and only if $u$ is a viscosity subsolution and it satisfies: for every $x\in M$, we can find an absolutely continuous curve $\gamma:\,]-\infty,0]\to M$, with $\gamma(0)=x$, and for every $t\geq 0$,
\begin{equation}\label{inftycalibration}
u(\gamma(0))-u(\gamma(-t))=
\int_{-t}^0\mathfrak L(\gamma(s),\dot\gamma(s))+c(\mathfrak H)\,\dd s.
\end{equation}
\end{itemize}
\end{thm}

\begin{definition}\rm 
A curve $\gamma:[a,b] \to M$ is called {\em $u$-calibrated} for $\mathfrak L+c$, if
\begin{equation*}
u(\gamma(b))-u(\gamma(a))=\int_{a}^b\mathfrak L(\gamma(s),\dot\gamma(s))+c\,\dd s.
\end{equation*}	
\end{definition}
We remark that if $u$ is a viscosity subsolution, then for any interval $[c,d]\subset [a,b]$, the restriction $\gamma_{|[c,d]}$ is also $u$-calibrated for $\mathfrak L+c$.

The following lemma follows the same spirit as the proof of Lemma \ref{Lip2} below. 

\begin{lemme} \label{calLipsc} Suppose that ${\mathfrak H}:T^*M\to\R$ is a continuous Hamiltonian satisfying conditions (${\mathfrak H}1$) and (${\mathfrak H}2$), whose associated Lagrangian is $\mathfrak L:TM\to\R$. For every $c\in \R$, we can find a constant $K=K(c)>0$ such that: if $u:M\to\R$ is a viscosity subsolution of  ${\mathfrak H}(x,d_xu)=c$ and $\gamma:[a,b]\to M$ is $u$-calibrated for $\mathfrak L+c$, then $\gamma$ is Lipschitz continuous with Lipschitz constant $\leq K=K(c)$.
\end{lemme}

As a consequence, the calibrated curves in Theorem \ref{calibrationThm} are equi-Lipschitz.

For every $t>0$, we define the action functional $h_t:M\times M\to \R$ as follows:
\begin{equation*}\label{def h_t}
h_t(x,y):=\inf\left\{\int_{-t}^0 \big[ \mathfrak{L}(\gamma(s),\dot\gamma(s))+c(\mathfrak{H})\big]\,\dd s\  \Big |\
\gamma\in \D{AC}([-t,0];\,M),\,\gamma(-t)=x,\,\gamma(0)=y \right\},
\end{equation*}
where $\D{AC}([-t,0];\,M)$ is the family of absolutely continuous curves from $[-t,0]$ to $M$.

\begin{definition}\rm 
The {\em Peierls barrier} is the function $h:M\times M\to\R$ defined by
\begin{equation}\label{def h}
h(x,y):=\liminf_{t\to +\infty} h_t(x,y).
\end{equation}
\end{definition}

It satisfies the following propertie, see \cite{Fathi_Siconolfi2005, Davini_Zavidovique2013}:
\begin{prop}\label{prop h}\ \\
{\rm a)} The Peierls barrier
 $h$ is finite valued and Lipschitz
continuous.\\
\noindent{\rm b)} If $w$ is a viscosity subsolution of
${\mathfrak H}(x,d_xu)=c({\mathfrak H})$, then 
$$w(x)-w(y)\leq h(y,x)\quad\text{ for every $x,y\in M$}.$$
{\rm c)} For every fixed $y\in
    M$, the function $h(y,\cdot)$ is a viscosity solution of
${\mathfrak H}(x,d_xu)=c({\mathfrak H})$; For every fixed $y\in
    M$, the function  $-h(\cdot,y)$ is a viscosity subsolution of
${\mathfrak H}(x,d_xu)=c({\mathfrak H})$.
\end{prop}

The \emph{projected Aubry set} $\mathcal{A}_{\mathfrak{L}}$ for the Lagrangian ${\mathfrak L}$ is the closed set defined by
$$  \mathcal{A}_{\mathfrak{L}}:=\{x\in M \mid h(x,x)=0\}. $$
The projected Aubry set is a uniqueness set for the H-J equation. In fact, we have (see \cite{Fathi2012}):

\begin{prop}\label{Aubry_uniqueness}
Let $w_-$ and $w_+$ be a viscosity subsolution and supersolution of the equation $\mathfrak{H}(x, d_xu)=c(\mathfrak{H})$, respectively. If $w_-(x)\leq w_+(x)$ for all $x\in\mathcal{A}_{\mathfrak{L}}$, then $w_-\leq w_+$ on $M$.	
\end{prop}

We now proceed to define the Mather (minimizing) measures for the Lagrangian $\mathfrak L$. This notion was first introduced by J. N. Mather in the case of a $C^2$ Tonelli Lagrangian \cite{Mather1991}, where the minimizing measures were obtained by minimizing the integral of the Lagrangian with respect to a family of probability measures invariant with respect to the Euler-Lagrange flow. However, as realized by R. Ma\~n\'e \cite{Mane1996}, this minimizing problem yields the same minimizers if it is done to a larger  family of closed measures. This equivalent definition therefore can be adapted to the current paper's setting since the Euler-Lagrange flow cannot be defined, see also \cite{Fathi_Siconolfi2005}.
  
 \begin{definition}\label{def closed measure} \rm
A Borel probability measure $\tilde\mu$ on $TM$ is called {\em closed}  if it satisfies:
\begin{itemize}
\item[\rm a)] \quad $\displaystyle{\int_{TM} \|v\|_x\,\dd \tilde\mu(x,v)<+\infty}$;
\item[\rm b)] \quad$\displaystyle{\int_{TM} \dd_x\varphi(v)\, \dd \tilde\mu (x,v) =0}$\quad  for every $\varphi\in \CC^1(M)$.
\end{itemize}
\end{definition}
There is a way to obtain closed measures from Lipschitz curves: if $\gamma:[a,b]\to M$ is a Lipschitz curve, we define the probability measure 
$\tilde\mu_\gamma$ on $TM$ by
$$\int_{TM}f(x,v)\,\dd\tilde\mu_\gamma:=\frac1{b-a}\int_a^bf(\gamma(s),\dot\gamma(s))\, \dd s,\quad \textup{for every $f\in C(TM)$.}$$
The support of the measure, denoted by $\supp \tilde\mu_\gamma$, is contained in the closure of the graph of the speed curve of $\gamma$, that is
$$\supp \tilde\mu_\gamma=\overline{\{(\gamma(s),\dot\gamma(s))\mid s\in [a,b],\text{where $\dot\gamma(s)$ exists}\}}.$$
Since $\gamma$ is Lipschitz, if $K$ is a Lipschitz constant for $\gamma$, we get 
$$\supp \tilde\mu_\gamma\subset \{(x,v)\in TM\mid \lVert v\rVert_x\leq K\}.$$
Hence the measure $\tilde\mu_\gamma$ has compact support. It is not difficult to check the following lemma.
\begin{lemme}\label{WeakCompactness} Suppose $\gamma_n:[a_n,b_n]\to M, n\geq 0$ is a sequence of equi-Lipschitz curves with a common Lipschitz constant $K$, then 
$$\supp \tilde\mu_{\gamma_n} \subset \{(x,v)\in TM \mid \lVert v\rVert_x\leq K\},\text{ for all $n\geq 0$.}$$
In particular, the sequence $\tilde\mu_{\gamma_n}$ is relatively compact for the weak topology on the space of probability measures on $TM$. Therefore, if $b_n-a_n\to+\infty$, we can extract a subsequence 
$\tilde\mu_{\gamma_{n_i}}$ which weakly converges to a probability measure $\tilde \mu$ on $TM$.
Moreover, any such limit measure $\tilde \mu$ is a closed measure.
\end{lemme}
\begin{proof}
	The relative compactness is obvious. To verify that $\tilde\mu$ is closed, we observe that:  for any $C^1$ function $\Phi: M\to \R$,
\begin{align*}
	\int_{TM} \dd_x\Phi(v)\,\dd\tilde\mu(x,v)
=\lim_{n_i\to+\infty}\int_{TM}\dd_x\Phi(v)\,\dd \tilde\mu_{\gamma_{n_i}}(x,v)
=&\lim_{n_i\to+\infty}\frac1{b_{n_i}-a_{n_i}}\int_{a_{n_i}}^{b_{n_i}} \frac{\dd }{\dd s } \Phi(\gamma(s))\,\dd s\\
=&\lim_{n_i\to+\infty}\frac{\Phi(b_{n_i})-\Phi(\gamma(a_{n_i}))}{b_{n_i}-a_{n_i}} =0,
\end{align*}
since $\Phi$ is bounded on the compact manifold $M$. This proves the desired result.
\end{proof}

\begin{prop}\label{Propertiesclosedmeas}
The following holds:
\begin{equation*}
	\min_{\tilde\mu}\int_{TM}{\mathfrak L}(x,v)\, \dd\tilde\mu(x,v)= -c(\mathfrak H)
\end{equation*}
where $\tilde\mu$ varies in the set of closed probability measures. Hence, 
	$
		\int_{TM}{\mathfrak L}(x,v)+c({\mathfrak H})\, \dd\tilde\mu(x,v) \geq 0,
	$ for any closed measure $\tilde\mu$.
\end{prop}

\begin{definition}\rm 
A closed probability measure $\tilde\mu$ on $TM$ is called a {\em Mather measure} for the Lagrangian ${\mathfrak L}$ if
$$\int_{TM}{\mathfrak L}(x,v)+c({\mathfrak H})\, \dd\tilde\mu(x,v)=0.$$
The set of all Mather measures of ${\mathfrak L}$ is denoted by $\Mis({\mathfrak L})$. A {\em projected Mather measure} is a probability measure  on $M$ of the form $\mu=\pi\tilde\mu$, where $\tilde\mu\in \Mis({\mathfrak L})$ and $\pi:TM\to M$ is the canonical projection.
\end{definition}

The set ${\mathfrak L}$ is non-empty. The Mather set $\widetilde{\cal M}_{\mathfrak L}\subset TM$ defined as 
$$\widetilde{\cal M}_{\mathfrak L}:=\overline{\bigcup_{\tilde\mu\in \Mis({\mathfrak L})}\supp \tilde\mu}$$
is compact.  This also implies that the set of Mather measures is compact for the weak topology on the space of probability measures on $TM$. We denote  by ${\cal M}_{\mathfrak L}=\pi \circ \widetilde{\cal M}_{\mathfrak L}$ the projected Mather set. It is a classical conclusion in Aubry-Mather theory that $ {\cal M}_{\mathfrak L}\subset {\cal A}_{\mathfrak L}$, and $\cal A_{\mathfrak L}\backslash{\cal M}_{\mathfrak L}\neq \emptyset$ could happen in rather wide occasions.

Finally, we summarize the properties of Mather measures. 
\begin{prop}\label{PropertiesMather} 
Suppose that $\mathfrak L:TM\to \R$ is a Lagrangian satisfying (${\mathfrak L}1$) and (${\mathfrak L}2$). Then,
\begin{itemize}
\item The set $\Mis({\mathfrak L})$ is not empty, i.e. there always exists a closed probability measure $\tilde \mu$
such that $\int_{TM}\mathfrak L(x,v)+c(\mathfrak H)\, \dd\tilde\mu(x,v)=0$.
\item The Mather set $\widetilde{\cal M}_{\mathfrak L}=\overline{\bigcup_{\tilde\mu\in \Mis({\mathfrak L})}\supp \tilde\mu}$
is compact. In addition, ${\cal M}_{\mathfrak L}\subset {\cal A}_{\mathfrak L}$.
\item The set $\Mis({\mathfrak L})$ is compact for the weak topology on the space of probability measures on $TM$.
\end{itemize}
\end{prop}
All the results above can be applied to any one of the Lagrangian $L^r:TM\to\R,r\in\R$ introduced in Section \ref{section_Intr}.

\section{Properties of discounted solutions}\label{section_bounded}
We now consider a Lagrangian  $L:TM\times \R\to\R$ satisfying conditions {\bf(L0)}, {\bf(L1)} and {\bf(L2)}. The associated Hamiltonian $H:T^*M\times \R\to\R$ therefore satisfies conditions {\bf(H0)}, {\bf(H1)} and {\bf(H2)}.

By Theorem \ref{ExistenceViscHJSansu},  we can find a global viscosity solution $\hat u: M \to\R$ of  
\begin{equation}\label{HJAStCr}H^0(x,d_x\hat u)=c(H^0)
\end{equation}
with $H^0(x,p)=H(x,p,0)$. Set $\hat u_+=\hat u-\min_M\hat u\geq 0$. Since $\hat u_+$ is also a global viscosity solution of \eqref{HJAStCr} and
$H(x,p,u)$ is non-decreasing in $u$, it is not difficult to show that $\hat u_+$ is a viscosity supersolution of \eqref{HJlambda} for all $\lambda>0$. In the same way, $\hat u_-=\hat u-\max_M \hat u\leq 0$ is a viscosity subsolution of \eqref{HJlambda} for all $\lambda>0$. 

Since $\hat u_-\leq\hat u_+$, using the Perron method (see for instance \cite{Ishii1987}) we can find a continuous viscosity solution to \eqref{HJlambda} for every $\lambda>0$. More precisely, the function defined by
\begin{equation*}
u_\lambda(x):=\sup\{w(x) \mid \hat u_-\leq w\leq \hat u_+, w \text{~is a continuous viscosity subsolution of~} \eqref{HJlambda} \}.
\end{equation*}	
is indeed a viscosity solution. Therefore, we conclude that:    

\begin{prop}\label{prop_exsitulam}
For every $\lambda >0$, we can find a continuous viscosity solution $u_\lambda:M\to\R$ of the discounted Hamilton--Jacobi equation \eqref{HJlambda} such that $\hat u -\max_M\hat u\leq  u_\lambda\leq \hat u -\min_M\hat u$.
\end{prop}
\begin{rem}
However, equation \eqref{HJlambda} may fail to have the uniqueness property of solutions because by condition {\bf(H0)} the map $u\mapsto H(x,p,u)$ is not necessarily strictly increasing. For more on uniqueness of the solutions, see Section \ref{section_uniqueness}.
\end{rem}

Next, we will interpret a solution of equation \eqref{HJlambda} as a solution of the Hamilton--Jacobi equation of a Hamiltonian to which we can apply the results of Section \ref{sectionNon-u}. Given a continuous function $\psi:M\to\R$, we can define a new Hamiltonian ${\bf H}^\psi:T^*M\to\R$ by
\begin{equation}\label{Def-bfH-u}
{\bf H}^\psi(x,p):=H(x,p, \psi(x)).
\end{equation} 
Since $H$ satisfies conditions {\bf(H1)} and {\bf(H2)}, using the continuity of $u$, it is not difficult to show that ${\bf H}^\psi$ satisfies conditions (${\mathfrak H}1$) and (${\mathfrak H}2$). Accordingly, the Lagrangian ${\bf L}^\psi:TM\to\R$ associated to ${\bf H}^\psi$ is given by 
\begin{equation}\label{Def-bfL-u}
{\bf L}^\psi(x,p)=L(x,v,\psi(x)).
\end{equation}  
Then we can easily verify the following conclusion:
\begin{prop}\label{uLambdaKamFaible} Any viscosity solution
$u_\lambda: M\to\R$ of  equation \eqref{HJlambda} is a viscosity solution
of the Hamilton--Jacobi equation
$${\bf H}^{\lambda u_\lambda}(x,d_xu_\lambda)=c(H^0).$$
As a consequence, the critical value $c({\bf H}^{\lambda u_\lambda})$ of the Hamiltonian ${\bf H}^{\lambda u_\lambda}: T^*M\to \R $ satisfies
\[c({\bf H}^{\lambda u_\lambda})=c(H^0).\]
\end{prop} 
Since  ${\bf L}^{\lambda u_\lambda}$ is the Lagrangian associated  to the Hamiltonian ${\bf H}^{\lambda u_\lambda}$, the next Proposition follows  from Lemma \ref{Lip0} and Theorem \ref{calibrationThm}.

\begin{prop}\label{calibrationThmLambda} 
Every viscosity solution  $u_\lambda:M\to\R$ of  \eqref{HJlambda} is Lipschitz on $M$. Moreover:
\begin{itemize}
\item[(1)] For every absolutely continuous curve $\gamma:[a,b]\to M$, we have
\begin{equation}\label{dominationLambda}
u_\lambda(\gamma(b))-u_\lambda(\gamma(a))\leq 
\int_a^b\big[  L(\gamma(s),\dot\gamma(s),  \lambda u_\lambda(\gamma(s))+c(H^0) \big]\,\dd s.
\end{equation}
\item[(2)] For every $x\in M$, we can find an absolutely continuous curve $\gamma:\,]-\infty,0]\to M$, with $\gamma(0)=x$, and 
\begin{equation}\label{inftycalibrationLambda}
u_\lambda(\gamma(0))-u_\lambda(\gamma(-t))=
\int_{-t}^0\big[ L(\gamma(s),\dot\gamma(s), \lambda  u_\lambda(\gamma(s))+c(H^0)\big]\,\dd s,
\end{equation}
for every $t\geq 0$.
\end{itemize}
\end{prop}

As previously a curve $\gamma:[a,b] \to M$ such that
\begin{equation}\label{calibration2}
u_\lambda(\gamma(b))-u_\lambda(\gamma(a))= 
\int_{a}^b\big[L(\gamma(s),\dot\gamma(s),  \lambda  u_\lambda(\gamma(s))+c(H^0)\big]\,\dd s,
\end{equation}
is called {\em $u_\lambda$-calibrated}. In such a case, for any interval $[c,d]\subset [a,b]$, the restriction
$\gamma_{|[c,d]}$ is also $u_\lambda$-calibrated.

Applying Propositions \ref{Propertiesclosedmeas} and \ref{PropertiesMather} to the Lagrangian ${\bf L}^{\lambda u_\lambda}$, we obtain:
\begin{prop}\label{PropertiesMatherLambda}
If $u_\lambda:M\to\R$ is a continuous viscosity solution of \eqref{HJlambda}, then we have:
\begin{itemize}
\item $\int_{TM}\big[  L(x,v,\lambda u_\lambda(x))+c(H^0) \big]\, \dd\tilde\mu \geq 0$, for all closed measures $\tilde\mu$ on $TM$.
\item There exists a closed probability measure $\tilde \nu$
such that $\int_{TM}\big[   L(x,v,\lambda u_\lambda(x))+c(H^0)  \big]\, \dd\tilde\nu=0$.
\end{itemize}
\end{prop}

\section{An important Lemma}\label{section_a_lemma}
This section is devoted to proving Lemma \ref{Gen-equibounded}, a very useful conclusion in the following sections. As previously we consider the Lagrangian $L:TM\times \R\to\R$ satisfying conditions {\bf (L0)}, {\bf(L1)} and {\bf (L2)}, and  the associated Hamiltonian $H:T^*M\times \R\to\R$.

Given two continuous functions $\varphi,\psi:M\to \R$, as done above we define two Hamiltonians
${\bf H}^\varphi,{\bf H}^\psi:T^*M\to\R$  by
$${\bf H}^\varphi(x,p)=H(x,p,\varphi(x))~\text{ and }~{\bf H}^\psi(x,p)=H(x,p,\psi(x)).$$
Their associated Lagrangians ${\bf L}^\varphi,{\bf L}^\psi:TM\to\R$ are given by
$${\bf L}^\varphi(x,v)=L(x,v,\varphi(x))~\text{ and }~{\bf L}^\psi(x,v)=L(x,v,\psi(x)).$$

\begin{lemme}\label{Gen-equibounded} Suppose that $u:M\to\R $ is a subsolution of 
$${\bf H}^\varphi(x,d_xu)=c$$
and $w:M\to\R $ is a solution of 
$${\bf H}^\psi(x,d_xw)=c.$$
Then one of the following two holds:
\begin{enumerate}
\item[(1)] The maximum of $u-w$ can be attained at a point $x_{\max}\in M$ where $\varphi(x_{\max})-\psi(x_{\max})\leq 0$.
\item[(2)] We can find a Lipschitz curve $\gamma : ]-\infty,0]\to M$ which is both $u$-calibrated for 
${\bf L}^\varphi+c$ and $w$-calibrated for ${\bf L}^\psi+c$, such that for all $t\in[0,+\infty[$,
\begin{align*}
u(\gamma(-t))- w(\gamma(-t))=\max(u-w), \qquad \varphi(\gamma(-t))-\psi(\gamma(-t))>0.
\end{align*}
\end{enumerate}
Moreover,  if (1) does not hold, denoting by $K<+\infty$ a Lipschitz constant 
for the Lipschitz curve $\gamma:]-\infty,0]\to M$ obtained in (2),
we can find
a \emph{closed} measure $\tilde \mu$ on $TM$, whose support, $\supp \tilde\mu$, is contained in the compact subset $\{(x,v)\in TM \mid \lVert v\rVert_x\leq K\}$,
such that
$$\int_{TM}\big[ L(x,v, \varphi(x))+c\big]\,\dd\tilde\mu(x,v)=\int_{TM} \big[L(x,v, \psi(x))+c\big]\,\dd\tilde\mu(x,v)=0,$$
and for all $(x,v)\in \supp \tilde\mu$,
$$u(x)-w(x)=\max_M(u-w),\ \  \varphi(x)-\psi(x)>0.$$
\end{lemme} 
\begin{proof}
Pick $x_0\in M$ such that
$u(x_0)- w(x_0)=\max_M(u- w)$. If $\varphi(x_0)-\psi(x_0)\leq 0$, then alternative (1) of the lemma holds.

Suppose now that $\varphi(x_0)-\psi(x_0)> 0$. Since $w$ is a solution of 
${\bf H}^\psi(x,d_xw)=c$,
 by Theorem \ref{calibrationThm} and Lemma \ref{calLipsc}, we can find a Lipschitz (with Lipschitz constant $K<\infty$) curve $\gamma:\,]-\infty,0]\to M$, with 
$\gamma(0)=x_0$, which is $w$-calibrated for ${\bf L}^\psi$.  In particular,
\begin{equation}\label{inftycalibrationv}
w(x_0)-w(\gamma(-t))=
\int_{-t}^0L(\gamma(s),\dot\gamma(s), \psi(\gamma(s)))+c\,\dd s,\text{ for every $t\geq 0$.}
\end{equation}
Since $\varphi(x_0)-\psi(x_0)> 0$, we can define $t_0\in [0, +\infty]$ by
\begin{equation}\label{deft0}t_0=\sup\{t\geq 0\mid \varphi(\gamma(-s))-\psi(\gamma(-s))>0,\text{ for all $s\in[0,t]$}\}.
\end{equation}
By continuity of $s\mapsto \varphi(\gamma(-s))-\psi(\gamma(-s))$, 
we have $t_0>0$ and $(\varphi-\psi)(\gamma(-t_0))=0$ if $t_0$ is finite.

\underline{\bf Claim:} We have
\begin{equation}\label{Constance1gen}
u(x_0)- w(x_0)=u(\gamma(-t))- w(\gamma(-t)),\text{ for all $t\in[0,t_0[$},
\end{equation}
and $\gamma: [0,t_0[ \to M$ is not only 
 $w$-calibrated for ${\bf L}^\psi$ but also $u$-calibrated for ${\bf L}^\varphi$.
 Moreover
\begin{equation}\label{Constance2gen}
L\bigl(\gamma(-t),\dot\gamma(-t), \varphi(\gamma (-t))\bigr)=L\bigl(\gamma(-t),\dot\gamma(-t), \psi(\gamma (-t))\bigr),
\text{ for almost all $t\in[0,t_0[$}.
\end{equation}

We now prove the Claim.
In fact by condition {\bf(L0)}, since $\varphi(\gamma(-s))-\psi(\gamma(-s))>0$, for all $s\in[0,t_0[$, we obtain
\begin{equation}\label{Momogen}
L\bigl(\gamma(-s),\dot\gamma(-s), \varphi(\gamma (-s))\bigr)\leq L\bigl(\gamma(-s),\dot\gamma(-s), \psi(\gamma (-s))\bigr),
\text{ for almost all $s\in[0,t_0[$}.
\end{equation}
Hence for any $t\in [0,t_0[$, taking inequality \eqref{Momogen} above together with equality \eqref{inftycalibrationv} and Theorem \ref{calibrationThm} applied to the viscosity subsolution of ${\bf H}^\varphi(x,d_xu)=c$, we obtain
\begin{equation}\label{BunchInegen}
\begin{aligned}
 u(x_0)-u(\gamma(-t))&\leq \int_{-t}^0\big[L\bigl(\gamma(s),\dot\gamma(s), \varphi(\gamma (s))\bigr)+c\big]\,\dd s\\
&\leq \int_{-t}^0\big[L\bigl(\gamma(s),\dot\gamma(s), \psi(\gamma (-s))\bigr)+c\big]\,\dd s\\ 
&= w(x_0)-w(\gamma(-t)).
\end{aligned}
\end{equation}
On the other side, recall that  $u(x_0)- w(x_0)=\max_M(u- w)$, which implies 
\[u(x_0)- u(x)\geq w(x_0)- w(x),\  \ \textup{for all $x\in M$}.\] In particular, we obtain
 $ u(x_0)-u(\gamma(-t))\geq w(x_0)-w(\gamma(-t))$, which forces all inequalities
 in \eqref{BunchInegen} to be equality. This thus implies that the curve $\gamma: [0,t_0[ \to M$   
 is also $u$-calibrated for ${\bf L}^\varphi$, and
 \begin{equation*}
 w(x_0)-w(\gamma(-t))=u(x_0)-u(\gamma(-t)), \, \textup{for all $t\in[0,t_0[$},
 \end{equation*}
so it verifies equality \eqref{Constance1gen}.  Moreover, since all inequalities in \eqref{BunchInegen} are indeed equality, it follows that
 $$\int_{-t}^0\big[L\bigl(\gamma(s),\dot\gamma(s), \varphi(\gamma (s))\bigr)+c\big]\,\dd s=
\int_{-t}^0\big[L\bigl(\gamma(s),\dot\gamma(s), \psi(\gamma (-s))\bigr)+c\big]\,\dd s$$
for all $t\in[0,t_0[$, which together with \eqref{Momogen} implies 
$$L\bigl(\gamma(-s),\dot\gamma(-s), \varphi(\gamma (-s))\bigr)=
L\bigl(\gamma(-s),\dot\gamma(-s), \psi(\gamma (-s))\bigr),\text{ for almost every $s\in [0,t[$.}$$
Since $t\in [0,t_0[$ is arbitrary, this yields equality \eqref{Constance2gen} in the Claim, 
which finally finishes establishing the Claim.

$\bullet$ If $t_0$ is finite, then by continuity $\max_M(u- w)=u(x_0)- w(x_0)=u(\gamma(-t_0))- w(\gamma(-t_0))$
and as we observed above $(\varphi-\psi)(\gamma(-t_0))=0$. Again alternative (1) of the lemma holds in that case.

$\bullet$ We will now consider the case $t_0=+\infty$. In this situation, the Claim above and the definition \eqref{deft0} of $t_0$ prove that the curve 
$\gamma: ]-\infty,0]$ satisfies alternative (2) of the lemma.

It remains to construct, the closed measure $\tilde\mu$, when alternative (1) of the lemma does not hold. In fact, when alternative (2) holds, we have 
\begin{equation}\label{positivitygen}(\varphi-\psi)(\gamma(-t))>0,\text{ for all $t\in[0,+\infty[$,}
\end{equation}
and by \eqref{Constance1gen}, we obtain
\begin{equation}\label{Constance1genbis}
u(x_0)- w(x_0)=u(\gamma(-t))- w(\gamma(-t)),\text{ for all $t\in[0,+\infty[$},
\end{equation}
Let $A\subset [0,+\infty[$ be the subset of full measure of $s\in[0,+\infty[$ where \eqref{Constance2gen} holds. Since $\gamma$ is Lipschitz, with Lipschitz constant $K$, the set
$$\widetilde C=\overline{\{(\gamma(-t),\dot\gamma(-t))\mid t\in A\}}\subset\{(x,v)\mid \lVert v\rVert_x \leq K\}$$
is compact. 

By continuity of $u,w,\varphi $ and $\psi$, we infer from \eqref{positivitygen} and \eqref{Constance1genbis} that
$$u(x)-w(x)=\max_M(u-w), \quad \varphi(x)-\psi(x)\geq0,\text{ for all $(x,v)\in \widetilde C$.}$$
As we are assuming alternative (1) false, we must have $ \varphi(x)-\psi(x)>0$ for all $(x,v)\in \widetilde C$. Therefore, 
\begin{equation*}
u(x)-w(x)=\max_M(u-w),\quad  \varphi(x)-\psi(x)>0,\text{ for all $(x,v)\in \widetilde C$.}
\end{equation*}
Since \eqref{Constance2gen} holds for all $s\in A$, from the definition of $\widetilde C$ it follows that
\begin{equation}\label{reduction}
L(x,v, \varphi(x))=L(x,v, \psi(x)),\text{ for all $(x,v)\in \widetilde C$.}
\end{equation}

We are now ready to construct the desired closed measure $\tilde \mu$ which satisfies $\supp \tilde\mu\subset\widetilde C$ and
\begin{equation}\label{lastRequirement}
\int_{TM}\big[ L(x,v, \psi(x))+c\big]\,\dd\tilde\mu(x,v)=\int_{TM}\big[ L(x,v, \varphi(x))+c\big]\,\dd\tilde\mu(x,v)=0.
\end{equation}
The construction is as follows: for every $t>0$, we can define a probability $\tilde \mu_t$ on $TM$ by
$$\int_{TM}f(x,v)\,\dd\tilde \mu_t:=\frac1t\int_{-t}^0f(\gamma(s),\dot\gamma(s))\,\dd s,\quad \textup{for any $f\in C(TM)$}.$$
Since the supports of the $\mu_t$ are all contained in the compact subset $\widetilde C\subset TM$, we can choose a sequence $t_n\to +\infty$ such that  
$\tilde \mu_{t_n}\to \tilde\mu$ in the weak topology. 

Clearly, $\supp \tilde\mu \subset\widetilde C$, and by Lemma \ref{WeakCompactness} this limit $\tilde\mu$ is a closed measure. Moreover, $\tilde \mu$ also satisfies
\begin{align*}
	\int_{TM} \big[L(x,v, \psi(x))+c\big]\,\dd\tilde\mu(x,v)
=&\lim_{n\to+\infty}\frac1{t_n}\int_{-t_n}^0\big[L\bigl(\gamma(s),\dot\gamma(s),\psi(\gamma(s))\bigr)+c\big]\,\dd s\\
=& \lim_{n\to+\infty}\frac{w(x_0)-w(\gamma(-t_n))}{t_n}=0,
\end{align*}
where we have used \eqref{inftycalibrationv} and the fact that $w$ is bounded on the compact manifold $M$. Hence, \eqref{lastRequirement} follows immediately from \eqref{reduction}. This finally finishes the proof. 
\end{proof}

\section{The \texorpdfstring{$u_\lambda$}{}'s are equibounded}\label{section_equibd}

Even though the Hamilton--Jacobi equation \eqref{HJlambda}, $\lambda>0$, may not have the uniqueness property of solutions under conditions {\bf (L0)}--{\bf(L4)}, we can still show that all possible solutions are uniformly bounded from below and above.

\begin{thm}\label{equibounded} Suppose that the Lagrangian $L:TM\times \R\to\R$ satisfies conditions {\bf (L0)}, {\bf(L1)}, {\bf(L2)}, {\bf(L3)} and {\bf(L4)}.
If $\hat u: M \to\R$ is a global viscosity solution of  
$$H^0(x,d_x\hat u)=c(H^0).$$
Then for every global continuous viscosity solution $u_\lambda$ of \eqref{HJlambda}, with $\lambda >0$, we have
$$ \hat u-\max_M\hat u\leq u_\lambda\leq \hat u-\min_M\hat u.$$
In particular, the solutions of \eqref{HJlambda}, for all $\lambda>0$, are equibounded.
\end{thm}

We split the proof of Theorem \ref{equibounded} in two lemmas, one where we show the upper bound and the other one where we show the lower bound. The key of the proof is to make use of Lemma \ref{Gen-equibounded}.


\begin{lemme}\label{Majoration}
 For every global continuous viscosity solution $u_\lambda$
of \eqref{HJlambda}, with $\lambda >0$, the function $u_\lambda- \hat u$ achieves its maximum at a point 
$x_{\max}\in M$ where $u_\lambda\leq 0$. 

Therefore, $u_\lambda\leq \hat u-\min_M\hat u$ for every global continuous viscosity solution $u_\lambda$ of \eqref{HJlambda}, with $\lambda >0$.
\end{lemme}
\begin{proof} 
The last part of the lemma follows from the first part, since 
\begin{equation*}
	u_\lambda-\hat u \leq \max_M (u_\lambda- \hat u)=u_\lambda(x_{\max})- \hat u(x_{\max})\leq - \hat u(x_{\max})\leq -\min_M \hat u.
\end{equation*}

Let us now prove the first part. Note that the functions $u_\lambda$ and $\hat u$ are respectively viscosity solutions of
${\bf H}^{\lambda u_\lambda}(x,d_xw)=c(H^0)$ and of
${\bf H}^{0}(x, d_x w)=c(H^0)$.
Therefore, we can apply Lemma \ref{Gen-equibounded} to conclude the following dichotomy: 
\begin{enumerate}
\item[(I)]	Either the maximum of $u_\lambda- \hat u$ can be achieved at a point where $\lambda u_\lambda-0$ is $\leq 0$.
\item[(II)] Or we can find a closed measure $\tilde\mu$ with compact support on
$TM$ such that
$$u_\lambda(x)- \hat u(x)=\max_M(u_\lambda- \hat u)~ \text{ and } ~ \lambda u_\lambda(x)>0, \text{ for all 
$(x,v)\in \supp{\tilde \mu}$,}$$
and 
\begin{equation}\label{Min-lambda} \int_{TM}\big[L(x,v, \lambda u_\lambda(x))+c(H^0)\big]\,\dd\tilde\mu=\int_{TM}\big[L(x,v,0)+c(H^0)\big]\,\dd\tilde\mu=0.
\end{equation}
\end{enumerate}

Case {\rm (I)} implies right away the lemma. 

If Case {\rm (II)} happens, to finish the proof of the lemma, it suffices to show that the existence of this measure $\tilde \mu$ leads to a contradiction.

Note that the second equality in \eqref{Min-lambda} implies that $\tilde \mu $ is a Mather  measure for $L^0$. Moreover, since $\lambda u_\lambda(x)>0$ for all $(x,v)\in \supp{\tilde \mu}$, condition {\bf(L0)} implies that
$$L(x,v,0)\leq L(x,v, t)\leq L(x,v, \lambda u_\lambda(x)), \text{ for all $(x,v,t)$ with $(x,v)\in \supp{\tilde \mu}$ and $t\in[0,\lambda u_\lambda(x)]$}.$$
This together with the first equality in \eqref{Min-lambda} implies that
$$L(x,v,0)=L(x,v, t)=L(x,v, \lambda u_\lambda(x)) , \text{ for all $(x,v,t)$ with
$(x,v)\in \supp{\tilde \mu}$ and $t\in[0,\lambda u_\lambda(x)]$}.$$
Since the partial derivative $\partial L/\partial u(x,v,0)$ exists (by condition {\bf(L3)}), using again that $\lambda u_\lambda(x)>0$ for all 
$(x,v)\in \supp{\tilde \mu}$, we conclude that
\begin{equation*}	
\frac{\partial L}{\partial u}(x,v,0)=0,\text{ for all 
$(x,v)\in \supp{\tilde \mu}$},
\end{equation*}
which in turn implies
$$\int_{TM}\frac{\partial L}{\partial u}(x,v,0)\,\dd\tilde\mu=0.$$
This contradicts condition {\bf(L4)}, since we have observed above that $\tilde \mu$ is a Mather measure for $L^0$.
\end{proof}

\begin{lemme}\label{Minoration} 
For every continuous viscosity solution $u_\lambda$
of \eqref{HJlambda}, with $\lambda >0$, the minimum of $ u_\lambda-\hat u$ is achieved at a point
where $u_\lambda$ is $\geq 0$.

Therefore, 
$u_\lambda\geq \hat u-\max_M\hat u$ for every  continuous viscosity solution $u_\lambda$ of \eqref{HJlambda}, with $\lambda >0$.
\end{lemme}
\begin{proof} The proof is very similar to that of Lemma \ref{Majoration}, exchanging the roles of $\hat u$
and $u_\lambda$.
\end{proof}

\section{ The \texorpdfstring{$u_\lambda$}{}'s and their calibrating curves are equi-Lipschitz}\label{section_equiLip}
Throughout this section, we assume that the Lagrangian satisfies conditions {\bf (L0)}, {\bf(L1)}, {\bf(L2)}, {\bf(L3)} and {\bf(L4)}.

\begin{lemme}\label{Lip1} For every finite $\lambda_0>0$, the family
of viscosity solutions of equation \eqref{HJlambda}, for $0<\lambda\leq \lambda_0$,
is an equi-Lipschitz family of functions on $M$.
\end{lemme} 
\begin{proof} From Theorem \ref{equibounded}, we can find a constant ${\mathfrak C}>0$ such that for any continuous viscosity solution $u_\lambda$ of \eqref{HJlambda},
\begin{equation*}
-{\mathfrak C}\leq u_\lambda\leq {\mathfrak C},
\end{equation*}
which implies
\begin{equation}\label{lambda0mathfrrakC}-\lambda_0{\mathfrak C}\leq\lambda u_\lambda\leq \lambda_0{\mathfrak C}, \ \ \text{for all $0<\lambda\leq \lambda_0$.}\end{equation}
From condition {\bf(L0)}, we therefore obtain
$$H(x,p,-\lambda_0{\mathfrak C})\leq H(x,p,\lambda u_\lambda), ~\text{for all $(x,p)\in T^*M$ 
and all $0<\lambda\leq \lambda_0$.}$$
Hence, all viscosity solutions $u_\lambda$ of \eqref{HJlambda}, for $0<\lambda\leq \lambda_0$, are viscosity subsolutions of the following equation
\begin{equation}\label{HJ10}H^{-\lambda_0{\mathfrak C}}(x,d_xu)=H(x,d_xu, -\lambda_0{\mathfrak C})=c(H^0).
\end{equation}
Since $H^{-\lambda_0{\mathfrak C}}(x,p)$ satisfies conditions (${\mathfrak H}1$) and (${\mathfrak H}2$), by Lemma \ref{Lip0} the family of viscosity subsolutions of
\eqref{HJ10} is equi-Lipschitz. This gives the desired result.
\end{proof}

As a consequence, by Theorem \ref{equibounded} together with Proposition \ref{prop_exsitulam} and Lemma \ref{Lip1}, we have thus proved Theorem \ref{mythm0}. We can also derive the following Lipschitz property on $u_\lambda$-calibrated curves, which will be used many times in the subsequent sections. 
\begin{lemme}\label{Lip2} 
For every finite $\lambda_0>0$, we can find a constant 
$K=K(\lambda_0)$ such that: if $u_\lambda$ is a continuous viscosity solution of \eqref{HJlambda} with $0<\lambda\leq\lambda_0$, then every $u_\lambda$-calibrated curve $\gamma:[a,b]\to M$ has Lipschitz constant $\leq K=K(\lambda_0)$.
\end{lemme}
\begin{proof}
By Lemma \ref{Lip1}, we can find common Lipschitz constant $K_0$ for all viscosity solutions
of \eqref{HJlambda}, with $0<\lambda\leq\lambda_0$.

By \eqref{lambda0mathfrrakC}, obtained in the proof of Lemma \ref{Lip1}, we can find a constant
${\mathfrak C}\geq 0$ such that
\begin{equation*}
-\lambda_0{\mathfrak C}\leq\lambda u_\lambda\leq \lambda_0{\mathfrak C}, ~\text{for all $0<\lambda\leq \lambda_0$.}
\end{equation*}
Since the Lagrangian $L$ satisfies condition {\bf(L0)}, for any viscosity solution $u_\lambda$ of  
 \eqref{HJlambda} with $0<\lambda\leq \lambda_0$, we obtain
$$L(x,v, \lambda u_\lambda(x))\geq L(x,v, \lambda_0 {\mathfrak C}), ~\text{ for all $(x,v)\in TM$.}$$
Let $\gamma:[a,b]\to M$ be $u_\lambda$-calibrated. If $[c,d]\subset [a,b]$, then
$\gamma_{|[c,d]}$ is  also $u_\lambda$-calibrated. Therefore
\begin{align*}
K_0 \, \dist(\gamma(c),\gamma(d))&\geq u_\lambda(\gamma(d))-u_\lambda(\gamma(c))\\
&= \int_{c}^d\big[L(\gamma(s),\dot\gamma(s), \lambda u_\lambda(\gamma(s)))+c(H^0)\big]\,\dd s\\
&\geq \int_{c}^d\big[L(\gamma(s),\dot\gamma(s), \lambda_0 {\mathfrak C})+c(H^0)\big]\,\dd s.
\end{align*}
By \eqref{SuperLinL},
$$ L(\gamma(s),\dot\gamma(s), \lambda_0 {\mathfrak C})\geq (K_0 +1)\lVert\dot\gamma(s)\rVert_{\gamma(s)}
+C(K_0+1,\lambda_0 {\mathfrak C}).$$
Therefore, integrating the last inequality between $c$ and $d$ and combining with the previous inequalities, we obtain
\begin{align*}
K_0\, \dist(\gamma(c),\gamma(d))&\geq (K_0 +1)\int_{c}^d\lVert\dot\gamma(s)\rVert_{\gamma(s)}\,\dd s
+C(K_0+1,\lambda_0 {\mathfrak C})(d-c)\\
&\geq  (K_0 +1)\,\dist(\gamma(c),\gamma(d))
+C(K_0+1,\lambda_0 {\mathfrak C})(d-c).
\end{align*}
Hence, 
\[  \frac{\dist(\gamma(c),\gamma(d))}{d-c}\leq -C(K_0+1,\lambda_0 {\mathfrak C}),~\textup{for any $[c,d]\subset [a,b]$,}\] 
so the lemma follows immediately by taking
$K=-C(K_0+1,\lambda_0 {\mathfrak C})$.
\end{proof}

\section{Proof of the convergence result}\label{section_mainproof}
The goal of this section is to show that the family of all  continuous viscosity solutions 
$u_{\lambda}: M\to \R,\lambda >0$, of equation \eqref{HJlambda} converges, as $\lambda\to 0^+$, to a particular solution of the critical equation $H^0(x, d_xu)=c(H^0)$.  We will also establish two characterizations of the limit solution.

The Lagrangian $L:TM\times\R\to\R$ considered in this section will always satisfy conditions {\bf (L0)}, {\bf(L1)}, {\bf(L2)}, {\bf(L3)} and {\bf(L4)}.

\subsection{Proof of Theorem \ref{main}}
We begin with the following lemma.
\begin{lemme}\label{Maj1} 
 Fix a compact subset $S\subset TM$. For every probability measure $\tilde\mu$ on $TM$ with $\supp \tilde\mu\subset S$,
and every continuous function $u:M\to\R$ and any $\lambda\in\R$, we have
\begin{equation}\label{MajMeas}\begin{gathered}\left\lvert \int_{TM}L(x,v,\lambda u(x))\,\dd\tilde\mu(x,v)- \int_{TM}L(x,v,0)\,\dd\tilde\mu(x,v)
-\lambda\int_{TM} u(x)\frac{\partial L}{\partial u}(x,v,0)\,\dd\tilde\mu(x,v)\right\rvert\\
\leq 
\lvert\lambda\rvert \lVert u\rVert_\infty\eta_S(\lvert\lambda\rvert \lVert u\rVert_\infty),
\end{gathered}
\end{equation}
where $\eta_S:[0,+\infty[\to [0,+\infty[$ is given by condition {\bf (L3)}.
\end{lemme}
\begin{proof} Condition {\bf (L3)} provides modulus of continuity $\eta_S$ such that
 $$\left\lvert L(x,v,u)-L(x,v,0)-u\frac{\partial L}{\partial u}(x,v,0)\right\rvert\leq \lvert u\rvert \eta_S(\lvert u\rvert),
 \text{ for all $(x,v,u)\in S\times \R.$}$$
Since  $\eta_S\geq 0$ and $\eta_S$ is non-decreasing, we obtain
$$\left\lvert L(x,v,\lambda u(x))-L(x,v,0)-\lambda u(x)\frac{\partial L}{\partial u}(x,v,0)\right\rvert
\leq \lvert\lambda\rvert \lVert u\rVert_\infty\eta_S(\lvert\lambda\rvert \lVert u\rVert_\infty),
 \text{ for all $(x,v)\in S.$}$$
Since $\tilde\mu$ is a probability measure with $\supp \tilde\mu\subset S$, we can integrate the previous inequality with respect to $\tilde\mu$ to obtain the desired inequality. 
\end{proof}

Due to Theorem \ref{equibounded} and Lemma \ref{Lip1}, any sequence of viscosity solutions $u_{\lambda}$ of equation \eqref{HJlambda}, with $0<\lambda\leq 1$, is equibounded and equi-Lipschitz. Therefore, by the Ascoli-Arzel\`a theorem this sequence is relatively compact in the topology of uniform convergence on $M$. The next result gives some constraints on the possible accumulation points of $\{u_\lambda\}_{\lambda>0}$ in terms of Mather measures of $L^0$. 

\begin{cor} \label{cor_subpositive}
Suppose that the sequence of positive numbers $\lambda_n$ tends to $0$ and that $u_{\lambda_n}:M\to \R$ are solutions of equation {\rm (HJ$_{\lambda_n}$)}, then 
$$\liminf_n \int_{TM} u_{\lambda_n}(x)\frac{\partial L}{\partial u}(x,v,0)\,\dd\tilde\mu(x,v)\geq 0,
\text{ for all $\tilde\mu\in \Mis({L^0})$.}$$
In particular, if $u_{\lambda_n}$ converges uniformly to some function $u:M\to\R$, then
$$ \int_{TM} u(x)\frac{\partial L}{\partial u}(x,v,0)\,\dd\tilde\mu(x,v)\geq 0,
\text{ for all Mather measures $\tilde\mu$ of $L^0$.}$$
\end{cor}
\begin{proof}
Let $\eta_S$ be the modulus given by condition {\bf(L3)} with $S= \widetilde{\cal M}_{L^0}=\overline{\bigcup_{\tilde\mu\in \Mis({L^0})}\supp \tilde\mu}$ the Mather set of $L^0$. By Lemma \ref{Maj1}, for every Mather measure $\tilde\mu\in\Mis(L^0)$, we have
\begin{equation}\label{intLlamn}
\begin{aligned}
\int_{TM}L(x,v,\lambda_n u_{\lambda_n}(x))\,\dd\tilde\mu(x,v)\leq 
 \int_{TM}L(x,v,0)\,\dd\tilde\mu(x,v)
+ &\lambda_n  \int_{TM} u_{\lambda_n}(x)\frac{\partial L}{\partial u}(x,v,0)\,\dd\tilde\mu(x,v)\\
&+ \lambda_n \lVert u_{\lambda_n}\rVert_\infty\eta_S(\lambda_n \lVert u_{\lambda_n}\rVert_\infty).
\end{aligned}
\end{equation}
Since $\tilde\mu$ is a closed measure, by Proposition \ref{PropertiesMatherLambda} we have 
\[\int_{TM}\big[ L(x,v,\lambda_n u_{\lambda_n}(x))+c(H^0)\big] \,\dd\tilde\mu(x,v)
\geq 0.\]
Moreover,  $\int_{TM}\big[ L(x,v,0)+c(H^0)\big] \,\dd\tilde\mu(x,v)=0$  since $\tilde\mu$ is a Mather  measure for $L^0$.
Hence, \eqref{intLlamn} implies
$$0\leq\lambda_n\int_{TM} u_{\lambda_n}(x)\frac{\partial L}{\partial u}(x,v,0)\,\dd\tilde\mu(x,v)
+
\lambda_n \lVert u_{\lambda_n}\rVert_\infty\eta_S(\lambda_n \lVert u_{\lambda_n}\rVert_\infty).$$
Dividing this last inequality by $\lambda_n>0$ and letting $\lambda_n\to 0$, we get
$$\liminf_n \int_{TM} u_{\lambda_n}(x)\frac{\partial L}{\partial u}(x,v,0)\,\dd\tilde\mu(x,v)\geq 0.\qed$$
\def\qed{}
\end{proof}
For our purpose, we define a continuous function $u_0:M\to \R$ by
\begin{equation}\label{def_u0}
	u_0(x):=\sup_{w\in \mathcal{S}_0} w(x)
\end{equation}
where $\mathcal{S}_0$ denotes the set of viscosity subsolutions $w$ of \eqref{HJ000} such that 
\begin{equation}\label{strcodn}
\int_{TM} w(x)\frac{\partial L}{\partial u}(x,v,0)\,\dd\tilde \mu(x,v)\geq 0, \quad\textup{for all Mather measures $\tilde\mu\in\Mis(L^0)$.}
\end{equation}
By Proposition \ref{subsol_properties}, $u_0$ is also a viscosity  subsolution of equation \eqref{HJ000}. We will see later that $u_0$ is indeed a viscosity solution. 

In the remainder of this section, we aim to show that $u_0$ is the only accumulation point of $u_{\lambda}$ as $\lambda\to 0^+$. By the stability of the notion of viscosity solution, every possible accumulation point $v_0$ of $u_{\lambda}$ is a viscosity solution of \eqref{HJ000}, then Corollary \ref{cor_subpositive} implies that $v_0\leq u_0$. Thus, it remains to show that very possible accumulation point of $u_{\lambda}$ as $\lambda\to 0^+$ is $\geq u_0$.


We will need the following lemma.
\begin{lemme}\label{essentiel} For any given constants $c'$ and $c$ satisfying $c'<c$ and
\[c>\sup_{\tilde\mu\in \Mis({L^0})}\int_{TM}\frac{\partial L}{\partial u}(x,v,0)\,\dd\tilde\mu\geq \inf_{\tilde\mu\in \Mis({L^0})}\int_{TM}\frac{\partial L}{\partial u}(x,v,0)\,\dd\tilde\mu>c'\]
we can find $\lambda_0=\lambda_0(c,c')>0$ and $T_0=T_0(c,c')>0$ such that any curve $\gamma:[a,b]\to M$ with $b-a\geq T_0$, which is $u_\lambda$-calibrated for a viscosity solution $u_\lambda$ of \eqref{HJlambda} with $0<\lambda\leq\lambda_0$, satisfies
\begin{equation}\label{twoccineq}
c' <\frac{1}{b-a}\int_{a}^{b} \frac{\partial L}{\partial u}(\gamma(s),\dot\gamma(s),0)\,\dd s<c.
\end{equation}
\end{lemme}
\begin{proof} 
 Let us first prove the right hand side inequality of \eqref{twoccineq}. We prove it by contradiction. If it is false, we can find a sequence $\lambda_n\in ] 0, 1]$, 
with $\lambda_n\to 0$, a viscosity solution $u_{\lambda_n}$ of  
 (HJ$_{\lambda_n}$), and  a $u_{\lambda_n}$-calibrated curve
$\gamma_n:[a_n,b_n]\to M$, with $b_n-a_n\to\infty$ such that
\begin{equation}\label{Tonton}
\frac1{b_n-a_n}\int_{a_n}^{b_n} \frac{\partial L}{\partial u}(\gamma_n(s),\dot\gamma_n(s),0)\,\dd s\geq c.
\end{equation}

We then define a probability measure $\tilde \mu_n$ on $TM$ by 
$$\int_{TM} f(x,v) \, \dd\tilde \mu_n :=\frac1{b_n-a_n}\int_{a_n}^{b_n}f(\gamma_n(s),\dot\gamma_n(s))\,\dd s, \text{ for $f\in C_c(TM,\R)$.}$$

By Lemma \ref{Lip2}, all the measures $\tilde\mu_{n}$ have support in the compact set $\{(x,v)\mid \lVert v\rVert_x\leq K=K(1)\}$.
Therefore, extracting a subsequence if necessary, we can assume that $\tilde\mu_{n}$ converges weakly to
$\tilde\mu$. 

By Lemma \ref{WeakCompactness}, $\tilde\mu$ is a closed measure.
Then, using the fact that the curves $\gamma_n$ are $u_{\lambda_n}$-calibrated and that $u_{\lambda_n}$ are equibounded (by Theorem \ref{equibounded}), we can show that $\tilde\mu$ is a Mather measure for $L^0$. Thus $\int_{TM}\partial L/\partial u(x,v,0)\,\dd\tilde\mu<c$,
but \eqref{Tonton} implies $\int_{TM}\partial L/\partial u(x,v,0)\,\dd\tilde\mu\geq c$, which leads to a contradiction. 
We have thus proved the right hand side inequality of \eqref{twoccineq}. 

As for the left hand side inequality of \eqref{twoccineq}, it can be proved in the same fashion as above.
\end{proof}

Note that we are assuming that {\bf(L4)} holds, namely 
$$\int_{TM}\frac{\partial L}{\partial u}(x,v,0)\,\dd\tilde\mu(x,v)<0,\text{ for all Mather measures $\tilde\mu\in\Mis(L^0)$.}$$
Since $\Mis(L^0)$ is compact, we can find $\epsilon>0, \epsilon'>0$ and $\epsilon<\epsilon'$, such that 
\begin{equation}\label{boundepsilon}
	-\epsilon'<\inf_{\tilde\mu\in \Mis({L^0})}\int_{TM}\frac{\partial L}{\partial u}(x,v,0)\,\dd\tilde\mu\leq \sup_{\tilde\mu\in \Mis({L^0})}\int_{TM}\frac{\partial L}{\partial u}(x,v,0)\,\dd\tilde\mu<-\epsilon.
\end{equation}
As a corollary of Lemma \ref{essentiel}, we obtain the following property.

\begin{cor}\label{cortwoestim}
We can find $\lambda_0>0$ and $T_0>0$ such that for any $u_{\lambda}$-calibrated curve
$\gamma_\lambda: ]-\infty,0] \to M$, with $\lambda\in ]0,\lambda_0]$, one has
\begin{enumerate}
	\item [\rm (i)]for any $t\in]-\infty, -T_0]$,
\begin{equation}\label{L_uleqepsilon}
\epsilon' t\leq \int_{t}^0\frac{\partial L}{\partial u}(\gamma_\lambda(s),\dot\gamma_\lambda(s),0)\,\dd s\leq \epsilon t.
\end{equation}
As a consequence, 
\begin{equation}\label{elamL_u0}
	 e^{\lambda\int_{-\infty}^0\frac{\partial L}{\partial u}(\gamma_\lambda(s),\dot\gamma_\lambda(s),0)\,\dd s}=0.
\end{equation}
    \item [\rm (ii)] for any $T\geq T_0$,
    \begin{equation}\label{integ_0Tlam}
     \frac{e^{-\lambda\epsilon' T_0}-e^{-\lambda\epsilon' T}}{\lambda\epsilon'}	\leq\int_{-T}^0e^{\lambda\int_{t}^0\frac{\partial L}{\partial u}(\gamma_\lambda(s),\dot\gamma_\lambda(s),0)\,\dd s}\,\dd t\leq T_0+\frac{1}{\lambda \epsilon}.
\end{equation}
In particular,   
\begin{align}\label{integ_0infty}
	\frac{e^{-\lambda\epsilon' T_0}}{\lambda\epsilon'}\leq  \int_{-\infty}^0e^{\lambda\int_{t}^0\frac{\partial L}{\partial u}(\gamma_\lambda(s),\dot\gamma_\lambda(s),0)\,\dd s}\,\dd t\leq  T_0+\frac{1}{\lambda\epsilon}.
\end{align}

\end{enumerate}
\end{cor}
\begin{proof}
By \eqref{boundepsilon}, we can apply Lemma \ref{essentiel} with $c=-\epsilon$ and $c'=-\epsilon'$ to find  two constants $\lambda_0>0$ and $T_0>0$ such that property (i) holds. 

We now turn to verify property (ii). By conditions {\bf(L0)} and {\bf(L3)}, we have $\frac{\partial L}{\partial u}(x,v,0)\leq 0$. This together with \eqref{L_uleqepsilon} implies that for every $T\geq T_0$, 
\begin{align*}
	\int_{-T}^0e^{\lambda\int_{t}^0\frac{\partial L}{\partial u}(\gamma_\lambda(s),\dot\gamma_\lambda(s),0)\,\dd s}\,\dd t
	=&\int_{-T_0}^0e^{\lambda\int_{t}^0\frac{\partial L}{\partial u}(\gamma_\lambda(s),\dot\gamma_\lambda(s),0)\,\dd s}\,\dd t+\int_{-T}^{-T_0}e^{\lambda\int_{t}^0\frac{\partial L}{\partial u}(\gamma_\lambda(s),\dot\gamma_\lambda(s),0)\,\dd s}\,\dd t\\
	\leq & \int_{-T_0}^0 e^{0} \,\dd t+\int_{-T}^{-T_0}e^{\lambda\epsilon t}\,\dd t\\
	= &T_0 +\frac{e^{-\lambda \epsilon T_0}-e^{-\lambda \epsilon T}}{\lambda \epsilon}\leq T_0+\frac{1}{\lambda \epsilon},
\end{align*}
so it proves  the right hand side inequality of \eqref{integ_0Tlam}.  On the other side, by \eqref{L_uleqepsilon} we obtain

\begin{equation*}
	\int_{-T}^0e^{\lambda\int_{t}^0\frac{\partial L}{\partial u}(\gamma_\lambda(s),\dot\gamma_\lambda(s),0)\,\dd s}\,\dd t\geq \int_{-T}^{-T_0}e^{\lambda\int_{t}^0\frac{\partial L}{\partial u}(\gamma_\lambda(s),\dot\gamma_\lambda(s),0)\,\dd s}\,\dd t \geq \int_{-T}^{-T_0}e^{\lambda\epsilon' t}\,\dd t=\frac{e^{-\lambda\epsilon' T_0}-e^{-\lambda\epsilon' T}}{\lambda\epsilon'}
\end{equation*}
which finally proves \eqref{integ_0Tlam}. As a consequence, inequality \eqref{integ_0infty} follows by sending $T\to +\infty$.
\end{proof}

For $x\in M$ and a viscosity solution $u_\lambda$ of equation \eqref{HJlambda} with $\lambda\in]0,\lambda_0]$, we fix a $u_{\lambda}$-calibrated curve $\galax: ]-\infty,0] \to M$ with $\galax(0)=x$ given in Proposition \ref{calibrationThmLambda}. Indeed, by Lemma \ref{Lip2} the curves $\galax$ are Lipschitz with Lipschitz constant $\leq K$. Then, for our purpose, we define a probability measure $\tilde \mu^x_\lambda$ on $TM$ as the following: 
\begin{align}\label{probmeasure_defnition}
\int_{TM}f(y,v)\,\dd\tilde \mu^x_\lambda(y,v):=
\frac1{\int_{-\infty}^0e^{\lambda \int_{t}^0\frac{\partial L}{\partial u}(\galax(s),\dgalax(s),0)\,\dd s}\,\dd t}
\int_{-\infty}^0f(\galax(t),\dgalax(t))\,
e^{\lambda\int_{t}^0\frac{\partial L}{\partial u}(\galax(s),\dgalax(s),0)\,\dd s}\,\dd t
\end{align}
for every $f\in C(TM)$.  This measures $\mu^x_\lambda$ is well defined as a result of Corollary \ref{cortwoestim}.

\begin{prop}\label{prop_relativecompact}
The probability measures $\{\tilde\mu^x_\lambda\}_{\lambda\in]0,\lambda_0]}$ defined above have support contained in a common compact subset of $TM$. Hence they are relatively compact in the space of probability measures on $TM$. Moreover, if $\tilde\mu^x_{\lambda_n}$ weakly converges to a measure $\tilde\mu$ for some sequence $\lambda_n\to 0$, then $\tilde\mu$ is a Mather measure of $L^0$.	
\end{prop}

\begin{proof}
The first part is a direct consequence of Lemma \ref{Lip2}. Let us pass to prove the second part. Suppose that the sequence $\tilde\mu^x_{\lambda_n}$ weakly converges to $\tilde\mu$, then $\tilde\mu$ has compact support. 

\textbf{The measure $\tilde\mu$ is closed:}  note that for $\lambda\in ]0,\lambda_0]$ and  a $C^1$ function $\phi: M\to \R $, 
\begin{equation} \label{Dphiv}
\begin{aligned}
     \int_{TM} \dd_y\phi(v) \,\dd\tilde{\mu}^x_\lambda(y,v)
     =& \frac{\int_{-\infty}^0 e^{\lambda\int_{t}^0\frac{\partial L}{\partial u}(\galax(s),\dgalax(s),0)\,\dd s} \cdot  \frac{\dd}{\dd t}\phi(\galax(t))\, \dd t}{ \int_{-\infty}^0e^{\lambda \int_{t}^0\frac{\partial L}{\partial u}(\galax(s),\dgalax(s),0)\,\dd s}\,\dd t}\\
     =&\frac{\phi(x)-\int_{-\infty}^0  \phi(\galax(t))\, \frac{\dd}{\dd t}e^{\lambda\int_{t}^0\frac{\partial L}{\partial u}(\galax(s),\dgalax(s),0)\,\dd s} \, \dd t}{ \int_{-\infty}^0e^{\lambda \int_{t}^0\frac{\partial L}{\partial u}(\galax(s),\dgalax(s),0)\,\dd s}\,\dd t}
\end{aligned}
\end{equation}
where for the last line we have used an integration by parts and \eqref{elamL_u0} in Corollary \ref{cortwoestim}. To handle the integral term in the numerator, we use the fact 
\begin{equation}\label{L_u_nonnege}
 \frac{\dd}{\dd t}e^{\lambda \, \int_{t}^0\frac{\partial L}{\partial u}(\galax(s),\dgalax(s),0)\,\dd s} =-\lambda \frac{\partial L}{\partial u}(\galax(t),\dgalax(t),0) \,  e^{\lambda\int_{t}^0\frac{\partial L}{\partial u}(\galax(s),\dgalax(s),0)\,\dd s}\geq 0.
\end{equation}
Hence,
\begin{equation}\label{vkjsdt}
\begin{aligned}
	\left| \int_{-\infty}^0  \phi(\galax(t))\frac{\dd}{\dd t}e^{\lambda\int_{t}^0\frac{\partial L}{\partial u}(\galax(s),\dgalax(s),0)\,\dd s} \, \dd t\right|
	\leq &\|\phi\|_\infty  \int_{-\infty}^0 \frac{\dd}{\dd t}e^{\lambda\int_{t}^0\frac{\partial L}{\partial u}(\galax(s),\dgalax(s),0)\,\dd s} \, \dd t\\
	\leq &\|\phi\|_\infty  \left(1-e^{\lambda\int_{-\infty}^0\frac{\partial L}{\partial u}(\galax(s),\dgalax(s),0)\,\dd s}\right)\\
	= &\|\phi\|_\infty
	\end{aligned}
	\end{equation}
where for the last equality we have used \eqref{elamL_u0}. Then, the last inequality together with \eqref{Dphiv}  and \eqref{integ_0infty} implies that
\begin{align*}
	\left|\int_{TM} \dd_y\phi(v) \,\dd\tilde{\mu}^x_\lambda(y,v)
      \right|\leq \lambda\epsilon' e^{\lambda\epsilon' T_0} (\|\phi\|_\infty+\|\phi\|_\infty)
 \longrightarrow 0
\end{align*}
as $\lambda\to 0$. Thus, if a sequence 
 $\tilde\mu^x_{\lambda_n}$ weakly converges to a measure $\tilde\mu$ as $\lambda_n\to 0$, we get
 \begin{align*}
	 \int_{TM} \dd_y\phi(v)\,\dd\tilde\mu(y,v) = \lim_{n\to \infty}\int_{TM}  \dd_y\phi(v) \,\dd\tilde{\mu}^x_{\lambda_n}(y,v)=0,
\end{align*}
which implies that $\tilde\mu$ is a closed measure.

\textbf{The measure $\tilde\mu$ is minimizing:} for each positive $\lambda\leq \lambda_0$, the function $t\in ]-\infty,0]\mapsto u_\lambda(\galax(t))$ is Lipschitz continuous and hence differentiable almost everywhere. Since $\galax$ is $u_\lambda$-calibrated,
\begin{align*}
	u_\lambda(x)-u_\lambda(\galax(t))=\int_t^0 L(\galax(s),\dgalax(s),\lambda u_\lambda(\galax(s)))+c(H^0)\,\dd s
\end{align*}
for any $t<0$, which in turn implies
\begin{align}\label{diffulamt}
	\frac{\dd}{\dd t}u_\lambda(\galax(t))=L(\galax(t),\dgalax(t),\lambda u_\lambda(\galax(t)))+c(H^0),\quad \textrm{a.e.}~t<0.
\end{align}

We want to show that $\int_{TM} L^0(y,v)+c(H^0)\,\dd \tilde{\mu}(y,v)=0$. For simplicity we set 
\begin{align*}
	\Delta_\lambda(t):=L(\galax(t),\dgalax(t),\lambda u_\lambda(\galax(t)))-L(\galax(t),\dgalax(t),0),
\end{align*}
which is a bounded Lebesgue measureable function. Then using \eqref{diffulamt} we derive that
\begin{align}
	&\int_{TM} L^0(y,v)+c(H^0)\,\dd\tilde{\mu}^x_\lambda(y,v)\nonumber\\
     =&  \frac{\int_{-\infty}^0  \big(L(\galax(t),\dgalax(t),0)+c(H^0)\big)\cdot e^{\lambda\int_{t}^0\frac{\partial L}{\partial u}(\galax(s),\dgalax(s),0)\,\dd s}\,\dd t }{ \int_{-\infty}^0e^{\lambda \int_{t}^0\frac{\partial L}{\partial u}(\galax(s),\dgalax(s),0)\,\dd s}\,\dd t}\nonumber\\
     =& \frac{\int_{-\infty}^0  \big(L(\galax(t),\dgalax(t),\lambda u_\lambda(\galax(t)))+c(H^0)-\Delta_\lambda(t) \big)\cdot e^{\lambda\int_{t}^0\frac{\partial L}{\partial u}(\galax(s),\dgalax(s),0)\,\dd s}\,\dd t }{ \int_{-\infty}^0e^{\lambda \int_{t}^0\frac{\partial L}{\partial u}(\galax(s),\dgalax(s),0)\,\dd s}\,\dd t}\nonumber\\
     =&  \frac{\int_{-\infty}^0  \left(\frac{\dd}{\dd t}u_\lambda(\galax(t))-\Delta_\lambda(t)\right)\cdot e^{\lambda\int_{t}^0\frac{\partial L}{\partial u}(\galax(s),\dgalax(s),0)\,\dd s}\,\dd t }{ \int_{-\infty}^0e^{\lambda \int_{t}^0\frac{\partial L}{\partial u}(\galax(s),\dgalax(s),0)\,\dd s}\,\dd t}\nonumber\\
     =&  \frac{u_\lambda(x)-\int_{-\infty}^0 u_\lambda(\galax(t)) 
     \frac{\dd}{\dd t}e^{\lambda\int_{t}^0\frac{\partial L}{\partial u}(\galax(s),\dgalax(s),0)\,\dd s}\,\dd t }{ \int_{-\infty}^0e^{\lambda \int_{t}^0\frac{\partial L}{\partial u}(\galax(s),\dgalax(s),0)\,\dd s}\,\dd t}
     - \frac{ \int_{-\infty}^0\Delta_\lambda(t)\, e^{\lambda\int_{t}^0\frac{\partial L}{\partial u}(\galax(s),\dgalax(s),0)\,\dd s}\,\dd t}{ \int_{-\infty}^0e^{\lambda \int_{t}^0\frac{\partial L}{\partial u}(\galax(s),\dgalax(s),0)\,\dd s}\,\dd t}
     \label{vnsnkai}
\end{align}
where we have used an integration by parts for the last line. Similar to \eqref{vkjsdt} we find that
\begin{align}\label{cawiuq1}
	\left|\int_{-\infty}^0 u_\lambda(\galax(t)) 
     \frac{\dd}{\dd t}e^{\lambda\int_{t}^0\frac{\partial L}{\partial u}(\galax(s),\dgalax(s),0)\,\dd s}\,\dd t\right|\leq \|u_\lambda\|_\infty.
\end{align}
Let $S=\{(x,v)\in TM \mid \lVert v\rVert_x\leq K\}$ with $K$ given in  Lemma \ref{Lip2}, by condition {\bf(L3)} we have a modulus of continuity $\eta_S$ such that  
\begin{equation}\label{cawiuq2}
	\|\Delta_\lambda(t)\|_{L^\infty} \leq \lambda\| u_\lambda\|_\infty \max_{ S}\left|\frac{\partial L}{\partial u}(x,v,0)\right|+\lambda \|u_\lambda\|_\infty \eta_S (\lambda \|u_\lambda\|_\infty)\leq \lambda C_1
\end{equation}
for some constant $C_1>0$ since $\{u_\lambda\}_{\lambda\in (0,\lambda_0]}$ is equibounded. Substituting \eqref{cawiuq1}--\eqref{cawiuq2} into \eqref{vnsnkai}, and using Corollary \ref{cortwoestim} (ii), we obtain
\begin{align*}
	\left|\int_{TM} L^0(y,v)+c(H^0)\,\dd\tilde{\mu}^x_\lambda(y,v)\right|\leq\frac{2\lambda\epsilon'\|u_\lambda\|_\infty}{e^{-\lambda \epsilon'T_0}}+\|\Delta_\lambda\|_{L^\infty}\leq\frac{2\lambda\epsilon'\|u_\lambda\|_\infty}{e^{-\lambda \epsilon'T_0}}+\lambda C_1\longrightarrow 0
\end{align*}
as $\lambda\to 0$. Hence, if a sequence 
 $\{\tilde\mu^x_{\lambda_n}\}_n$ weakly converges to a measure $\tilde\mu$, then
 \begin{align*}
	 \int_{TM} L^0(y,v)+c(H^0) \,\dd\tilde\mu(y,v) = \lim_{n\to \infty}\int_{TM} L^0(y,v)+c(H^0)\,\dd\tilde{\mu}^x_{\lambda_n}(y,v)=0.
\end{align*}
Therefore, we conclude that $\tilde\mu$ is a Mather measure.
\end{proof}

As a corollary, using Proposition \ref{prop_relativecompact} and condition {\bf(L4)} we can easily derive: 
\begin{cor}\label{cor_nonvanish}
We can find a suitably small $\lambda_1\in ]0,\lambda_0[$, such that for every $\lambda\in ]0,\lambda_1]$ and $u_\lambda$-calibrated curve $\gamma^x_\lambda:]-\infty,0]\to M$ with $\gamma^x_\lambda(0)=x$, 
\begin{align*}
	\int_{TM} \frac{\partial L}{\partial u}(y,v,0)\, \dd\tilde\mu^x_\lambda(y,v)< 0.
\end{align*}
\end{cor}

The next lemma will be crucial for the proof of Theorem \ref{main}.
\begin{lemme}\label{lem_ugeqsubsol}
	Let $w$ be any viscosity subsolution of the critical equation \eqref{HJ000}. For every $x\in M$ and $\lambda\in ]0, \lambda_1]$ where   $\lambda_1$ is given in Corollary \ref{cor_nonvanish}, we have 
	\begin{equation}\label{w_lamgeqsth}
	u_{\lambda}(x)\geq w(x)-\frac{\int_{TM} w(y)\,\frac{\partial L}{\partial u}(y,v,0)\, \dd\tilde\mu^x_\lambda(y,v)}{\int_{TM} \frac{\partial L}{\partial u}(y,v,0)\, \dd\tilde\mu^x_\lambda(y,v) }+R_\lambda(x),
	\end{equation}
	and the term $R_\lambda(x)$ satisfies 
	\[\lim_{\lambda\to 0^+}R_\lambda(x)=0.\]
\end{lemme}
\begin{rem}
	Due to Corollary \ref{cor_nonvanish}, the denominator on the right hand side of \eqref{w_lamgeqsth} is non-zero.
\end{rem}

\begin{proof}
Let  $w(x)$ be any  viscosity subsolution of equation \eqref{HJ000}. By a well-known approximation argument, see
for example \cite{Fathi2012}, for subsolutions of convex Hamilton--Jacobi equations, for each $\delta>0$ we can select  $w_\delta \in C^\infty(M)$	 such that $\|w_\delta-w\|_\infty\leq \delta$ and 
\begin{align*}
	H^0(y, \dd_yw_\delta)\leq c(H^0)+\delta \quad \textup{for all~} y\in M.
\end{align*}
Then, using the Fenchel inequality one gets
\begin{align}\label{fnuibiq}
	L^0(y,v)+c(H^0)\geq L^0(y,v)+H^0(y, \dd_yw_\delta)-\delta\geq \dd_y w_\delta(v)-\delta \quad \textup{for all~} (y,v)\in TM.
\end{align}

For the $u_{\lambda}$-calibrated curve $\galax: ]-\infty,0] \to M$, recalling \eqref{diffulamt} one has
\begin{align*}
	\frac{\dd}{\dd t}u_\lambda(\galax(t))=L(\galax(t),\dgalax(t),\lambda u_\lambda(\galax(t)))+c(H^0),\quad \textrm{a.e.}~t<0
\end{align*}
This, combined with \eqref{fnuibiq},  implies that for a.e. $t<0$
\begin{align}
	\frac{\dd}{\dd t}u_\lambda(\galax(t))\geq & 	\dd_{\galax(t)}w_\delta \big(\dgalax(t)\big)-\delta+L(\galax(t),\dgalax(t),\lambda u_\lambda(\galax(t)))-L^0(\galax(t),\dgalax(t))\nonumber\\
	=&  \frac{\dd}{\dd t}w_\delta(\galax(t))+  \lambda  u_\lambda(\galax(t)) \,\frac{\partial L}{\partial u}(\galax(t),\dgalax(t),0)
	  -\delta+\Omega_{\lambda,x}(t)\label{kdaniw}
\end{align}
where 
\[\Omega_{\lambda,x}(t)=L(\galax(t),\dgalax(t),\lambda u_\lambda(\galax(t)))-L^0(\galax(t),\dgalax(t))-
 \lambda  u_\lambda(\galax(t)) \,\frac{\partial L}{\partial u}(\galax(t),\dgalax(t),0).\]

Multiplying both sides of \eqref{kdaniw} by $e^{\lambda\int^0_t \frac{\partial L}{\partial u}(\galax(s),\dgalax(s),0)\, \dd s}$, we obtain that for a.e  $t<0$,
\begin{align*}
\frac{\dd}{\dd t}\left(u_\lambda(\galax(t))\, e^{\lambda\int^0_t \frac{\partial L}{\partial u}(\galax(s),\dgalax(s),0)\, \dd s} \right)	\geq\left( \frac{\dd}{\dd t} w_\delta(\galax(t))-\delta +\Omega_{\lambda,x}(t)\right)e^{\lambda\int^0_t \frac{\partial L}{\partial u}(\galax(s),\dgalax(s),0)\, \dd s}
\end{align*} 
Then, for any $T\geq T_0$ where $T_0$ is given in Corollary \ref{cortwoestim}, by integrating the above inequality over the interval $(-T,0]$ and using an integration by parts, we obtain
\begin{align*}
    u_\lambda(x)-u_\lambda(\galax(-T))\, e^{\lambda\int^0_{-T} \frac{\partial L}{\partial u}(\galax(s),\dgalax(s),0)\, \dd s} 	\geq & 
    w_\delta(x)-w_\delta(\galax(-T))\,e^{\lambda\int^0_{-T} \frac{\partial L}{\partial u}(\galax(s),\dgalax(s),0)\, \dd s} \\
    &-\int_{-T}^0   w_\delta(\galax(t)) \frac{\dd}{\dd t} \Big(e^{\lambda\int^0_t \frac{\partial L}{\partial u}(\galax(s),\dgalax(s),0)\, \dd s}\Big)\,\dd t
    \\
    &+\int^0_{-T}\left(\Omega_{\lambda,x}(t)-\delta\right)e^{\lambda\int^0_t \frac{\partial L}{\partial u}(\galax(s),\dgalax(s),0)\, \dd s}\,\dd t.
\end{align*}
Sending $\delta\to 0$ and using \eqref{L_uleqepsilon} yields 
\begin{align*}
    u_\lambda(x)\geq & 
    w(x) -\int_{-T}^0   w(\galax(t)) \frac{\dd}{\dd t} \Big(e^{\lambda\int^0_t \frac{\partial L}{\partial u}(\galax(s),\dgalax(s),0)\, \dd s}\Big)\,\dd t-(\|u_\lambda\|_\infty +\|w\|_\infty )\,e^{-\epsilon T}\\
    &\qquad \qquad\qquad\qquad+\int^0_{-T}\Omega_{\lambda,x}(t)\,e^{\lambda\int^0_t \frac{\partial L}{\partial u}(\galax(s),\dgalax(s),0)\, \dd s}\,\dd t
\end{align*}
Furthermore, by letting $T\to +\infty$ it follows that
\begin{align}\label{u_lam_ineq1}
    u_\lambda(x)\geq & 
    w(x)  -\underbrace{\int_{-\infty}^0   w(\galax(t)) \frac{\dd}{\dd t} \Big(e^{\lambda\int^0_t \frac{\partial L}{\partial u}(\galax(s),\dgalax(s),0)\, \dd s}\Big)\,\dd t}_{\mathbf{I}_\lambda}+
    \underbrace{\int^0_{-\infty}\Omega_{\lambda,x}(t)\,e^{\lambda\int^0_t \frac{\partial L}{\partial u}(\galax(s),\dgalax(s),0)\, \dd s}\,\dd t}_{R_\lambda(x)}
\end{align}
Let us now consider the integral $\mathbf{I}_\lambda$: using the measure $\tilde\mu^x_\lambda$ (see \eqref{probmeasure_defnition}) we infer
\begin{align}
	\mathbf{I}_\lambda=&-\lambda\int_{-\infty}^0   w(\galax(t))  \frac{\partial L}{\partial u}(\galax(t),\dgalax(t),0)e^{\lambda\int^0_t \frac{\partial L}{\partial u}(\galax(s),\dgalax(s),0)\, \dd s}\,\dd t \nonumber\\
	    =& -\lambda \int^0_{-\infty}e^{\lambda\int^0_t \frac{\partial L}{\partial u}(\galax(s),\dgalax(s),0)\, \dd s}\,\dd t\cdot\int_{TM}   w(y)  \frac{\partial L}{\partial u}(y,v,0)\,\dd\tilde\mu^x_\lambda(y,v).\label{mabfI_lam}
\end{align}
When $\lambda\in]0,\lambda_1]$, we can derive that 
\begin{align*}
	\lambda \int^0_{-\infty}e^{\lambda\int^0_t \frac{\partial L}{\partial u}(\galax(s),\dgalax(s),0)\, \dd s}\,\dd t
	=&\frac{\lambda \int^0_{-\infty}e^{\lambda\int^0_t \frac{\partial L}{\partial u}(\galax(s),\dgalax(s),0)\, \dd s}\,\dd t}{ \int^0_{-\infty}\frac{\dd}{\dd t}\left(e^{\lambda\int^0_t \frac{\partial L}{\partial u}(\galax(s),\dgalax(s),0)\, \dd s}\right)\,\dd t}\nonumber\\ 
	=&\frac{\int^0_{-\infty}e^{\lambda\int^0_t \frac{\partial L}{\partial u}(\galax(s),\dgalax(s),0)\, \dd s}\,\dd t}{- \int^0_{-\infty}\frac{\partial L}{\partial u}(\galax(t),\dgalax(t),0)\,e^{\lambda\int^0_t \frac{\partial L}{\partial u}(\galax(s),\dgalax(s),0)\, \dd s}\,\dd t}\nonumber\\
	=&\frac{-1}{ \int_{TM}\frac{\partial L}{\partial u}(y,v,0)\,\dd\tilde\mu^x_\lambda(y,v)}, 
\end{align*}
where the denominator is non-zero by Corollary \ref{cor_nonvanish}. Using this equality and \eqref{mabfI_lam} we conclude 
\begin{equation*}
	\mathbf{I}_\lambda=\frac{\int_{TM} w(y)\,\frac{\partial L}{\partial u}(y,v,0)\, \dd\tilde\mu^x_\lambda(y,v)}{\int_{TM} \frac{\partial L}{\partial u}(y,v,0)\, \dd\tilde\mu^x_\lambda(y,v) }.
\end{equation*}

To finish the proof it only remains to verify that $R_\lambda(x)\to$ as $\lambda\to 0$. Note that by Lemma \ref{Lip2} all calibrated curves $\galax$ have Lipschitz constant$\leq K$. We set $S=\{(x,v)\in TM \mid \lVert v\rVert_x\leq K\}$, then condition {\bf (L3)} provides a modulus of continuity $\eta_S$ such that
\[ \|\Omega_{\lambda,x}\|_\infty\leq \lambda\|u_\lambda\|_\infty \eta_S(\lambda\|u_\lambda\|_\infty).\]
This, combined with \eqref{integ_0infty}, yields
\begin{align*}
	|R_\lambda(x)|\leq \lambda\|u_\lambda\|_\infty \eta_S(\lambda\|u_\lambda\|_\infty) \int^0_{-\infty}e^{\lambda\int^0_t \frac{\partial L}{\partial u}(\galax(s),\dgalax(s),0)\, \dd s}\,\dd t
	\leq  (\lambda T_0+\epsilon^{-1})\, \|u_\lambda\|_\infty \, \eta_S(\lambda\|u_\lambda\|_\infty).
\end{align*}
As the modulus of continuity satisfies $\lim_{u\to 0^+} \eta_S(u)=0$ and the $u_\lambda$'s are equibounded, we immediately conclude that $\lim_{\lambda\to 0^+} R_\lambda(x) =0$. This completes the proof.
\end{proof}

We are now ready to finish the proof of Theorem \ref{main}:
\begin{proof}[Proof of Theorem \ref{main}]
Recalling \eqref{def_u0} we have defined a function $u_0:M\to \R$,
\begin{align*}
u_0(x)=\sup_{w\in \mathcal{S}_0} w(x),	
\end{align*}
where $\mathcal{S}_0$ denotes the set of subsolutions $w$ of equation \eqref{HJ000} such that $\int_{TM}w(y)\frac{\partial L}{\partial u}(y,v,0)\,\dd\tilde \mu(y,v)\geq0$, for all $\tilde\mu\in\Mis(L^0)$.

We will establish that for any sequence of viscosity solution $u_\lambda$ of \eqref{HJlambda}, it uniformly converges to $u_0$ as $\lambda\to 0^+$.
We have seen from Corollary \ref{cor_subpositive} that every possible  accumulation point of the sequence $u_{\lambda}$, as $\lambda\to 0^+$, is $\leq u_0$. If we prove the opposite inequality, the assertion follows.

Suppose that the sequence $\lambda_n\to 0^+$ and $u_{\lambda_n}$ converges uniformly to some function $v_0:M\to\R$. Let $x\in M$, and we consider a family of probability measures $\tilde\mu^x_{\lambda_n}$ given in \eqref{probmeasure_defnition}. By Proposition \ref{prop_relativecompact}, up to extracting a subsequence,  we may assume that $\tilde\mu^x_{\lambda_n}$ converges to a Mather measure $\tilde\mu$. Then for every subsolution $w\in \mathcal{S}_0$, we infer from Lemma \ref{lem_ugeqsubsol} that
\begin{equation}\label{finainequ}
v_0(x)\geq w(x)- \frac{\int_{TM} w(y)\,\frac{\partial L}{\partial u}(y,v,0)\, \dd\tilde\mu(y,v)}{\int_{TM} \frac{\partial L}{\partial u}(y,v,0)\, \dd\tilde\mu(y,v)}.
\end{equation}
Using condition {\bf (L4)} and the fact $w\in \mathcal{S}_0$ we find that 
\[\int_{TM} \frac{\partial L}{\partial u}(y,v,0)\, \dd\tilde\mu(y,v)<0,\qquad \int_{TM} w(y)\,\frac{\partial L}{\partial u}(y,v,0)\, \dd\tilde\mu(y,v)\geq 0,\] 
so \eqref{finainequ} implies that $v_0(x)\geq w(x)$. Since this inequality holds for all $w\in \mathcal{S}_0$, we conclude that $v_0\geq u_0$, which therefore completes the proof.  
\end{proof}

\subsection{An alternate formula for \texorpdfstring{$u_0$}{}}

We provide another formula for the limit solution $u_0$, which involves the Peierls barrier $h:M\times M\to \R$ of the Lagrangian $L^0$, as well as  Mather measures. This will prove our Theorem \ref{mythm_anotherformula}.

We start with the following lemma. 

\begin{lemme}\label{lm_onesi}
	For every $x\in M$ and $\tilde\mu\in  \Mis(L^0)$, we have 
	\begin{equation*}
		u_0(x)\leq \frac{1}{\int_{TM}\frac{\partial L}{\partial u}(y,v,0)\,\dd\tilde{\mu}(y,v)}\int_{TM}h(y,x)\frac{\partial L}{\partial u}(y,v,0)\,\dd\tilde{\mu}(y,v).
	\end{equation*}
\end{lemme}
\begin{proof}
Let $x\in M$. Since $u_0$ is a viscosity solution of $H^0(x, d_x u)=c(H^0)$, by Proposition \ref{prop h} we get $u_0(x)\leq u_0(y)+h(y,x)$ for all $y\in M$. Multiplying this inequality by $\partial L/\partial u(y,v,0)$ which is non-positive and integrating with respect to a Mather measure $\tilde\mu\in \Mis(L^0)$,  we derive that
\begin{align*}
	u_0(x)   \int_{TM} \frac{\partial L}{\partial u}(y,v,0) \,\dd\tilde{\mu}(y,v) & \geq \int_{TM} u_0(y) \frac{\partial L}{\partial u}(y,v,0)\,\dd\tilde{\mu}(y,v)+\int_{TM} h(y,x) \frac{\partial L}{\partial u}(y,v,0) \,\dd\tilde{\mu}(y,v)\\
	&\geq  \int_{TM} h(y,x) \frac{\partial L}{\partial u}(y,v,0) \,\dd\tilde{\mu}(y,v)
\end{align*}
where the last line follows from Corollary \ref{cor_subpositive}. By condition {\bf(L4)},  $\int_{TM} \frac{\partial L}{\partial u}(y,v,0) \,\dd\tilde{\mu}<0$, so the assertion follows.
\end{proof}

We define a function $\hat u_0: M\to \R$ by
\begin{equation*}
	\hat u_0(x):=\inf_{\tilde{\mu}\in \Mis(L^0) }\frac{\int_{TM}h(y,x)\frac{\partial L}{\partial u}(y,v,0)\,\dd\tilde{\mu}(y,v)}{\int_{TM}\frac{\partial L}{\partial u}(y,v,0)\,\dd\tilde{\mu}(y,v)}.
\end{equation*}
Then we will prove the following equality. 
\begin{thm}
	For every $x\in M$,  $u_0(x)=\hat u_0(x)$.
\end{thm}
\begin{proof}
	By Lemma \ref{lm_onesi}, we clearly have $u_0\leq\hat u_0$, so it remains to show that $u_0\geq\hat u_0$.

	We first claim that $\hat u_0$ is a viscosity subsolution of $H^0(x, d_xu)=c(H^0)$. Indeed, by Proposition \ref{prop h}, for each $y$, the function $x\mapsto h(y,x)$ is a viscosity solution of $H^0(x, d_xu)=c(H^0)$. For each $\tilde\mu \in \Mis(L^0)$, we define  $h_{\tilde\mu}:M\to \R$ by 
	\[h_{\tilde\mu}(x):=\frac{\int_{TM}h(y,x)\frac{\partial L}{\partial u}(y,v,0)\,\dd\tilde{\mu}(y,v)}{\int_{TM}\frac{\partial L}{\partial u}(y,v,0)\,\dd\tilde{\mu}(y,v)}. \]
	By Proposition \ref{subsol_properties}, $h_{\tilde\mu}$ is a subsolution of $H^0(x, d_xu)=c(H^0)$ since it is a convex combination of equi-Lipschitz viscosity solutions. Then, using Proposition \ref{subsol_properties} again,  we find that $\hat u_0$ is a subsolution since it is the infimum of a family of equi-Lipschitz subsolutions. Thus, the claim follows.
	
	For each $y\in M$, the function $-h(\cdot, y)$ is a subsolution of $H(x, d_xu)=c(H^0)$, see Proposition \ref{prop h}. We then define $U_y: M\to \R$ by
	\[U_y(x):=-h(x,y)+\hat u_0(y).\]
 Obviously, $U_y$ is a subsolution and we can check that
 \begin{equation*}
 	\frac{\int_{TM}U_y(x)\frac{\partial L}{\partial u}(x,v,0)\,\dd\tilde{\nu}(x,v)}{\int_{TM}\frac{\partial L}{\partial u}(x,v,0)\,\dd\tilde{\nu}(x,v)}\leq 0,\quad\textup{for all $\tilde\nu\in \Mis(L^0)$,}
 \end{equation*}
 which implies that $U_y$ satisfies condition \eqref{strcodn}. Thus, $u_0\geq U_y$ everywhere. In particular, for $y$ in the projected Aubry set $\mathcal{A}_{L^0}$, we have
	$u_0(y)\geq U_y(y)=\hat u_0(y)$ since $h(y,y)=0$.  We can conclude from 
	Proposition \ref{Aubry_uniqueness} that the inequality $u_0\geq \hat u_0$ holds on $M$. This finishes the proof.  
\end{proof}

\section{Global and Local uniqueness of the discounted solutions}\label{section_uniqueness}
In this section we provide a comparison result for equation \eqref{HJlambda} with $\lambda>0$. Under the hypotheses {\bf(L0)}--{\bf(L4)} it is not clear that we have uniqueness of solutions of equation \eqref{HJlambda} for each $\lambda>0$.

To be able to give a uniqueness property  for all $\lambda>0$, we will use condition {\bf(L5)}, that is
 \begin{itemize}[itemindent=1em]
 \item[\bf(L5)] For every $(x,v)\in TM$, the function $u\in \R  \mapsto L(x,v,u)$  is concave\footnote{This condition is equivalent to: for every $(x,p)\in T^*M$, the function $ u\in \R \mapsto H(x,p,u)$  is convex.}.
  \end{itemize}
  \begin{thm}[Global uniqueness] Suppose that the Lagrangian  $L:TM\times \R\to\R$ satisfies conditions {\bf (L0)}, {\bf(L1)}, {\bf(L2)}, {\bf(L3)}, {\bf (L4)} and {\bf(L5)}. Then for every $\lambda>0$  the Hamilton--Jacobi equation
\eqref{HJlambda} has a unique viscosity solution. 
\end{thm}
\begin{proof} It suffices to show that $\max_M(\tilde u_\lambda - u_\lambda)\leq 0 $ for any pair $\tilde u_\lambda, u_\lambda $ of 
 solutions of \eqref{HJlambda} with $\lambda>0$. We apply Lemma \ref{Gen-equibounded}, with $\varphi=\lambda \tilde u_\lambda,
\psi=\lambda u_\lambda, u=\tilde u_\lambda$ and $w=u_\lambda$. 

 If alternative (1) of Lemma 
\ref{Gen-equibounded} holds, then $\max_M(\tilde u_\lambda - u_\lambda)$ can be attained at a point where
$\lambda\tilde u_\lambda -\lambda u_\lambda$ is $\leq 0$. Since $\lambda>0$, we indeed have
$\max_M(\tilde u_\lambda - u_\lambda)\leq 0 $.

It remains to show that alternative (2) of Lemma \ref{Gen-equibounded} leads to a contradiction. Assume alternative (2) of  Lemma 
\ref{Gen-equibounded} happens, then we can find
a closed measure $\tilde \mu$ on $TM$, whose support $\supp \tilde\mu$ 
is compact, such that
\begin{equation}\label{811}\int_{TM}\big[ L(x,v, \lambda\tilde u_\lambda(x))+c({H^0}) ] \,\dd\tilde\mu(x,v)=\int_{TM} \big[ L(x,v, \lambda u_\lambda(x))+c(H^0)\big] \,\dd\tilde\mu(x,v)=0\end{equation}
and
$$\tilde u_\lambda(x)-u_\lambda(x)=\max_M(\tilde u_\lambda-u_\lambda) \ \ , \lambda\tilde u_\lambda(x)-\lambda u_\lambda(x)>0\ \ ,\text{ for all $(x,v)\in \supp \tilde\mu$.}$$
Since $\lambda>0$, we get $\tilde u_\lambda(x)>u_\lambda(x)$ whenever 
$(x,v)\in \supp \tilde\mu$. Using condition {\bf(L0)}, i.e. $L$ is non-increasing in $u$, we infer from \eqref{811} that
$$L(x,v,r)=L(x,v,\tilde u_\lambda(x)),\text{  for  all $(x,v)\in \supp \tilde\mu$,  
$r\in [\lambda u_\lambda(x),\lambda \tilde u_\lambda(x)]$.}$$
As the interval $ [\lambda u_\lambda(x),\lambda \tilde u_\lambda(x)]$ has non-empty interior, for all $(x,v)\in \supp \tilde\mu$, the concavity condition {\bf(L5)} implies that 
$$\max_{r\in \R}L(x,v,r)=L(x,v,\lambda\tilde u_\lambda(x)),\text{ for all $(x,v)\in \supp\tilde\mu$.}$$
Taken together with condition {\bf(L0)} this last condition implies that, for all $(x,v)\in \supp\tilde\mu$, 
we have
\begin{equation}\label{822}L(x,v,r)=L(x,v,\lambda\tilde u_\lambda(x)), \text{ for all $r\leq \lambda\tilde u_\lambda(x)$.}\end{equation}
We now fix $u_0<0$, such that $u_0\leq \min_M\lambda \tilde u_\lambda$. From \eqref{811} and
\eqref{822}, we obtain that
$$\int_{TM}\big[ L(x,v, u_0)+c({H^0})\big]\,\dd\tilde\mu(x,v)=\int_{TM} \big[L(x,v, \lambda\tilde u_\lambda(x))+c({H^0})\big] \,\dd\tilde\mu(x,v)=0.$$
Since $u_0<0$, condition {\bf(L0)} implies
$$0=\int_{TM} \big[L(x,v, u_0)+c({H^0})\big]\,\dd\tilde\mu(x,v)\geq  \int_{TM} \big[L(x,v, 0)+c({H^0}) \big] \,\dd\tilde\mu(x,v).$$
But since $\tilde\mu$ is a closed measure, we must have $ \int_{TM} \big[ L(x,v, 0)+c({H^0})\big] \,\dd\tilde\mu(x,v)\geq 0$. Hence
\begin{equation}\label{833}
\int_{TM} \big[ L(x,v, u_0)+c({H^0})\big] \,\dd\tilde\mu(x,v)= \int_{TM} \big[ L(x,v, 0)+c({H^0})\big]\,\dd\tilde\mu(x,v)=0.
\end{equation}
The second equality above implies that $\tilde\mu$ is a Mather measure for
$L^0$. Using again {\bf(L0)} together with $u_0<0$, the first equality of \eqref{833} implies that, 
$$L(x,v,u)=L(x,v,0), \text{ for all $(x,v)\in \supp \tilde\mu$, $u\in [u_0,0]$.}$$
Since $u_0<0$, we conclude that
$$\frac{\partial L}{\partial u}(x,v,0)=0, \text{ for all $(x,v)\in \supp \tilde\mu$,}$$
which contradicts condition {\bf(L4)}, because $\tilde\mu$ is a Mather measure for $L^0$. This finishes the proof.
\end{proof}

If we drop condition {\bf(L5)}, we can still prove uniqueness of solutions of 
\eqref{HJlambda} for small $\lambda>0$, provided that we strengthen condition {\bf(L3)} 
to condition {\bf(L3$'$)}, namely:
\begin{itemize}[itemindent=1em]
 \item[\bf(L3$'$)] For every $(x,v,u)\in TM\times \R$, the partial derivative $\partial L/\partial u(x,v,u)$ exists
 and is continuous on $TM\times \R$.
 \end{itemize}

 We will need the following lemma: 
\begin{lemme}\label{essentiel2} Suppose that the Lagrangian  $L:TM\times \R\to\R$  satisfies conditions {\bf (L0)}, {\bf(L1)}, {\bf(L2)}, {\bf(L3$'$)} and {\bf(L4)}. 
We can find  $\epsilon>0, \lambda_1>0$ and $T_0>0$ such that
any curve $\gamma:[a,b]\to M$, with $b-a\geq T_0$, which is $u_\lambda$-calibrated, where $u_\lambda$
is a solution of \eqref{HJlambda} and  $0<\lambda\leq\lambda_1$,
 satisfies
$$\int_{a}^{b} \frac{\partial L}{\partial u}(\gamma(s),\dot\gamma(s),\lambda u_\lambda(\gamma(s)))\,\dd s\leq -(b-a)\epsilon.$$
\end{lemme}
\begin{proof} 
From Theorem \ref{equibounded}, we can find ${\mathfrak C}$ a common bound for all $\lVert u_\lambda\rVert_\infty$. 

By Lemma \ref{Lip2}, we can find a constant  $K$ such that any curve 
$\gamma:[a,b]\to M$, which is $u_\lambda$-calibrated, with  $u_\lambda$
 a solution of \eqref{HJlambda} and $0<\lambda\leq 1$,
has Lipschitz constant $\leq K$.

By Lemma \ref{essentiel}, we can find  $\epsilon>0,\lambda_0>0$ and $T_0>0$ such that any curve 
$\gamma:[a,b]\to M$, which is $u_\lambda$-calibrated for ${\bf L}^{\lambda u_\lambda}+c(H^0)$, with  $u_\lambda$
 a solution of \eqref{HJlambda} and $0<\lambda\leq \lambda_0$, satisfies
\begin{equation}\label{ZuZu1}
\int_{a}^{b} \frac{\partial L}{\partial u}(\gamma(s),\dot\gamma(s),0)\,\dd s\leq -2(b-a)\epsilon.
 \end{equation}
 By the continuity of $\partial L/\partial u(x,v,u)$ on the compact set 
 $\{(x,v,u)\mid \lVert v\rVert_x\leq K, \lvert u\rvert\leq 1\}$, we can find a constant $\delta>0$ such that
\begin{equation}\label{ZuZu2}
\left\lvert \frac{\partial L}{\partial u}(x,v,u)-
 \frac{\partial L}{\partial u}(x,v,0)\right\rvert\leq \epsilon,
\end{equation}
for all $(x,v,u)\in TM\times \R$ with $\lVert v\rVert_x\leq K$ and $ \lvert u\rvert\leq \delta$.

If $\gamma:[a,b]\to M$, which is $u_\lambda$-calibrated curve, with  $u_\lambda$
 a solution of \eqref{HJlambda} and $0<\lambda\leq \min(\lambda_0,\delta/{\mathfrak C})=\lambda_1$, then $u_\lambda$
has Lipschitz constant $\leq K$ and $\lambda\lVert u\rVert_\infty\leq \delta$. Moreover, by \eqref{ZuZu1}
$$ \int_{a}^{b} \frac{\partial L}{\partial u}(\gamma(s),\dot\gamma(s),0)\,\dd s\leq -2(b-a)\epsilon,$$
and by \eqref{ZuZu1}
$$\left\lvert \frac{\partial L}{\partial u}(\gamma(s),\dot\gamma(s),\lambda u_\lambda(\gamma(s)))-
 \frac{\partial L}{\partial u}(\gamma(s),\dot\gamma(s),0)\right\rvert\leq \epsilon,$$
for almost every $s\in [a,b]$. Therefore,
$$ \int_{a}^{b} \frac{\partial L}{\partial u}(\gamma(s),\dot\gamma(s),\lambda u_\lambda(\gamma(s)))\,\dd s\leq -(b-a)\epsilon.\qed$$
\def\qed{}
\end{proof}

\begin{rem}
From the proof of Lemma \ref{essentiel2} we find that condition {\bf(L3$'$)} can be weakened to only hold on a neighborhood of $TM\times\{0\}$.
\end{rem}

\begin{thm}[Local uniqueness] 
Suppose that the Lagrangian  $L:TM\times \R\to\R$ satisfies conditions {\bf (L0)}, {\bf(L1)}, {\bf(L2)}, {\bf(L3$'$)} and {\bf(L4)}. Then there exists $\lambda_1>0$ such that the Hamilton--Jacobi equation
\eqref{HJlambda} has a unique viscosity solution for $0<\lambda\leq \lambda_1$. 
\end{thm}
\begin{proof} We choose the $\lambda_1$ and $T_0$ given by Lemma \ref{essentiel2}.
Again it suffices to show that $\max_M(\tilde u_\lambda - u_\lambda)\leq 0 $ for any pair $\tilde u_\lambda, u_\lambda $ of 
 solutions of \eqref{HJlambda} with $0<\lambda\leq \lambda_1$.
 
Again, we apply Lemma \ref{Gen-equibounded}, with $\varphi=\lambda \tilde u_\lambda,
\psi=\lambda u_\lambda, u=\tilde u_\lambda$ and $w=u_\lambda$. If alternative (1) of that Lemma 
\ref{Gen-equibounded} holds, then $\max_M(\tilde u_\lambda - u_\lambda)$ can be attained at a point where
$\lambda\tilde u_\lambda -\lambda u_\lambda$ is $\leq 0$. Since $\lambda>0$, we indeed have
$\max_M(\tilde u_\lambda - u_\lambda)\leq 0 $.

It remains to show that alternative (2) of Lemma \ref{Gen-equibounded} leads to a contradiction.

By alternative (2), we can find a Lipschitz curve $\gamma:\,]-\infty,0]\to M$, which is  $u_\lambda$-calibrated  for ${\bf L}^{\lambda u_\lambda}+c(H^0)$ such that:
\begin{enumerate}
\item [(a)] $\lambda \tilde u_\lambda(\gamma(t))-\lambda u_\lambda(\gamma(t))>0$, or equivalently
$\tilde u_\lambda(\gamma(t))>u_\lambda(\gamma(t))$, for all $t\in ]-\infty,0]$.

\item [(b)] $L(\gamma(t),\dot\gamma(t),\lambda \tilde u_\lambda(\gamma(t))=L(\gamma(t),\dot\gamma(t),\lambda u_\lambda(\gamma(t))$,  for almost all $t\in ]-\infty,0]$.
\end{enumerate}
By condition {\bf(L0)}, we infer from equality (b) above that for almost all $t\in ]-\infty,0]$, 
$$L(\gamma(t),\dot\gamma(t),r)=L(\gamma(t),\dot\gamma(t),\lambda u_\lambda(\gamma(t)),
\text{  for  all $r\in [\lambda u_\lambda(\gamma(t)),\lambda \tilde u_\lambda(\gamma(t))]$.}$$ 
Since by {\bf(L3$'$)}, the partial derivative $\partial L/\partial u$ exists everywhere on $TM\times \R$,
we conclude that
$$\frac{\partial L}{\partial u}(\gamma(t),\dot\gamma(t),\lambda u_\lambda(\gamma(t))=0,\text{  for almost all $t\in ]-\infty,0]$.}$$
Hence,
$$\int_{-T}^0\frac{\partial L}{\partial u}(\gamma(t),\dot\gamma(t),\lambda u_\lambda(\gamma(t))\,\dd t=0
\text{ for  all $T\geq 0$.}$$
This contradicts Lemma \ref{essentiel2} for $0<\lambda\leq \lambda_1$ and $T>T_0$.
\end{proof}

\section{An extension of the results}\label{section_an_extension}

In our theorems, it may be argued that $c(H^0)$ is rarely easy to compute and therefore that the equations \eqref{HJlambda} are not natural. In this section we explain how to apply the previous results to solutions of the equations 
\begin{equation}\tag{HJ$_\lambda^0$}\label{eq_hj0lam}
H(x,d_x u,\lambda u(x))=0.
\end{equation}

To this end we assume that the Lagrangian $L$ associated to $H$ verifies {\bf (L0)}, {\bf(L1)}, {\bf(L2)} and {\bf(L3$'$)}. Assume moreover the following:
\begin{itemize}[itemindent=1em]
 \item [\bf(L4$'$)] There exists a constant $c_0 \in \R$ such that the critical value $c(H^{c_0})=0$\footnote{ This is for example the case if there exists
 $C>0$ such that $H(x,0,C)\geq 0$ and $H(x,0,-C)\leq 0$ for all $x\in M$.} 
 and such that 
 $$\int_{TM}\frac{\partial L}{\partial u}(x,v,c_0)\,\dd\tilde\mu(x,v)<0,\text{ for all Mather measures 
 $\tilde\mu\in\Mis(L^{c_0})$}.$$
 \end{itemize}
 
 Hence, as previously, it can be proved, using Perron's method, that \eqref{eq_hj0lam} admits viscosity solutions for every $\lambda>0$. The first issue is to study the possible limits of $\lambda v_\lambda$ as $\lambda \to 0$. A hint is given by the next lemma:

\begin{lemme}\label{lem_exten}
For all $c_1\neq c_0$, we have $c(H^{c_1}) \neq c(H^{c_0})$.
\end{lemme}

\begin{proof}
We only sketch the proof as it uses ideas already exposed several times. Assume the existence of $c_1 \in \R$ such that  $c(H^{c_1}) = c(H^{c_0})=0$ and such that $c_0>c_1$, the other case is treated in an analogous way. Let $u_0 : M\to \R$ and  $u_1: M\to \R$ be, respectively, viscosity solutions of $H(x,d_x u_0, c_0) = 0$ and $H(x,d_x u_1, c_1) = 0$.

Let $x_0\in M$ be a point where $u_0-u_1$ reaches its maximum and $\gamma : ]-\infty , 0] \to M$ a Lipschitz curve that is $u_1$-calibrated for $L^{c_1}$ with $\gamma(0) = x_0$. It follows that for all $t>0$,

\begin{align*}
u_0(x_0) - u_0(\gamma(-t)) &\leq \int_{-t}^0 L(\gamma(s),\dot\gamma(s), c_0)\,\dd s \\
&\leq \int_{-t}^0 L(\gamma(s),\dot\gamma(s), c_1)\,\dd s = u_1(x_0) - u_1(\gamma(-t)) ,
\end{align*}
where we have used the monotonicity of $L$. Since $u_0(x_0)-u_1(x_0)=\max_M(u_0-u_1)$, it follows that all the previous inequalities are equalities and that 
$L(\gamma(s),\dot\gamma(s), c_0) = L(\gamma(s),\dot\gamma(s), c_1) $ for almost every $s\in  ]-\infty , 0]$ and finally that $\partial L / \partial u(\gamma(s),\dot\gamma(s), c_0) =0$  for almost every $s\in  ]-\infty , 0]$. Then, using a similar argument as in the proof of Lemma \ref{Gen-equibounded}  we can construct a Mather measure for $L^{c_0}$ contradicting condition {\bf(L4$'$)}.
\end{proof}

It can be easily seen that the map $c_1 \mapsto c(H^{c_1}) $ is non-decreasing. Lemma \ref{lem_exten} shows that it is strictly increasing at $c_0$. We now state the main result of this section:

\begin{thm}
For the Lagrangian $L :TM\times\mathbb R \rightarrow \mathbb R$ satisfying conditions {\bf (L0)}, {\bf(L1)}, {\bf(L2)}, {\bf(L3$'$)} and {\bf(L4$'$)}, 
there exists a function $v_0 : M\to \R$ that is a viscosity solution of $H(x,d_x v_0, c_0) = 0$ such that:
 For all $\lambda>0$, let $v_\lambda$ be a viscosity solution of \eqref{eq_hj0lam}, then the family $v_\lambda - c_0/\lambda$ uniformly converges to $v_0$ as $\lambda \to 0$.
\end{thm}

\begin{proof}
Let us define the Hamiltonian $\widetilde H : (x,p,u) \mapsto H(x,p,c_0+u)$. Then $0$ is the critical value of $\widetilde H(x,p,0)$. For $\lambda>0$ let us set $\tilde v_\lambda = v_\lambda - c_0/\lambda$, it is a solution of $\widetilde H(x, d_x \tilde v_\lambda , \lambda \tilde v_\lambda(x))=0$. Therefore, Theorem \ref{main} can be applied to $\widetilde H$ and gives the desired result.
\end{proof}

\section{A counterexample}\label{section_counterexample}

We conclude with a simple example showing the necessity of the non--degeneracy  hypothesis {\bf(L4)}.

Let us start with a $1$-periodic Hamiltonian of the pendulum on $\R$: 
\[H(x,p) = \frac12 |p|^2 + \cos(2\pi x).\] Its critical value is $c(H)=1$ (whether we see it as a $1$-periodic function  a $2$-periodic function or a $4$-periodic function).

Let $\alpha : \R \to \R$ be a non-negative continuous function that is supported on $[1/2,3/2]$ and positive on $]1/2,3/2[$. For each integer  $n\geq 2$, we will denote by $\alpha_n:\R\to\R$ the $n$-periodic function that coincides with $\alpha$ on $[0,n]$. We will denote by $H_n (x,p,u) = H(x,p) + \alpha_n(x) u$. 

If $\lambda>0$, let $u_\lambda : \R \to \R$ be a $2$-periodic function solution of $H_2 (x,d_x u_\lambda,\lambda u_\lambda(x))=1$. Such a function exists by Perron's method, with arguments already presented before. Such a function is differentiable at $x=2$ with $u_\lambda'(x)=0$.

We consider the equation $H(x, u'(x)) =1$, and pick two $2$-periodic solutions of this equation. The first one is actually the only $1$-periodic solution up to constants, and we only write it on $[0,1]$:
$$v^1 ( x) = 
\begin{cases}
\int_0^x \sqrt{ 2-2\cos(2\pi s) }\,\dd s ,& x\in[0,1/2], \\
\int_1^x -\sqrt{ 2-2\cos(2\pi s) }\,\dd s ,& x\in[1/2,1];
\end{cases}
$$
so $v^1(0)=v^1(1)=v^1(2)=0$.
The second one is $C^1$ and $2$-periodic, and we write it on $[0,2]$:

$$v^2(x)= 
\begin{cases}
\int_0^x \sqrt{ 2-2\cos(2\pi s) }\,\dd s, & x\in[0,1], \\
\int_2^x -\sqrt{ 2-2\cos(2\pi s) }\,\dd s, & x\in[1,2];
\end{cases}
$$
so $v^2 (0)=v^2(2)=0$.

We now construct $4$-periodic functions (we give the definition on $[0,4]$) that are solutions of $H_4(x,v'(x),\lambda v(x))=1$ on $M=\R/4\Z$, by gluing a piece of $u_\lambda$ with $v^1$ or $v^2$. More precisely, for $\lambda>0$, we define  
$$v^1_\lambda ( x) = 
\begin{cases}
u_\lambda(x) ,& x\in[0,2], \\
v^1(x)+u_\lambda(2) ,& x\in[2,4];
\end{cases}
\quad\qquad
v^2_\lambda ( x) = 
\begin{cases}
u_\lambda(x) ,& x\in[0,2], \\
v^2(x)+u_\lambda(2) ,& x\in[2,4].
\end{cases}
$$
Note that the gluing are differentiable at junction points.

With these constructions, it is obvious that solutions of $H_4(x,v'(x),\lambda v(x))=1$ do not necessarily converge as $\lambda \to 0$. Note that the Mather measures of $H_4^0$ are convex combinations of the Dirac measures $\delta_{(0,0)}$, $\delta_{(1,0)}$, $\delta_{(2,0)}$ and $\delta_{(3,0)}$ supported on hyperbolic equilibria.  However, the non--degeneracy condition {\bf(L4)} is only verified for the measure $\delta_{(1,0)}$.

\bibliography{mybib}
\bibliographystyle{alpha}

\end{document}